\documentclass[11pt]{article}
\usepackage[utf8]{inputenc}

\usepackage[width=16cm,height=22cm]{geometry}

\usepackage{titling}
\usepackage{amsmath,amsfonts,amssymb,amsthm,mathtools}
\usepackage{graphicx}
\usepackage{color}
\usepackage{comment}
\usepackage[colorlinks=true,allcolors=blue]{hyperref}
\usepackage{url}
\newcommand{\doi}[1]{\href{https://doi.org/#1}{doi:#1}}
\newcommand{\myurl}[1]{\href{#1}{#1}}
\usepackage{cite}
\usepackage{enumitem}

\newcounter{mycounter}
\newcommand{\mylabel}[1]{\refstepcounter{mycounter}\label{#1}}

\DeclareMathOperator{\arcosh}{arcosh}

\newcommand{\remainder}{\operatorname{rem}}
\newcommand{\al}{\alpha}
\newcommand{\be}{\beta}
\newcommand{\ga}{\gamma}

\newcommand{\eps}{\varepsilon}
\newcommand{\La}{\Lambda}
\newcommand{\ka}{\varkappa}
\newcommand{\la}{\lambda}
\newcommand{\ph}{\varphi}
\newcommand{\si}{\sigma}
\newcommand{\tht}{\vartheta}

\newcommand{\om}{\omega}

\newcommand{\expterm}{g}

\newcommand{\oldp}{\widetilde{p}}
\newcommand{\oldq}{\widetilde{q}}
\newcommand{\oldtht}{\widetilde{\tht}}
\newcommand{\oldE}{\widetilde{E}}
\newcommand{\oldM}{\widetilde{M}}

\newcommand{\bZ}{\mathbb{Z}}
\newcommand{\bC}{\mathbb{C}}
\newcommand{\bN}{\mathbb{N}}

\newcommand{\bR}{\mathbb{R}}

\newcommand{\enumber}{\operatorname{e}}
\newcommand{\imagunit}{\operatorname{i}}
\newcommand{\imunit}{\operatorname{i}}
\newcommand{\eqdef}{\coloneqq}

\newcommand{\SSS}{\Lambda}

\newcommand{\spectrum}{\operatorname{sp}}
\newcommand{\diag}{\operatorname{diag}}

\renewcommand{\phi}{\varphi}
\newcommand{\sign}{\operatorname{sign}}
\newcommand{\matr}[2]{\left[\begin{array}{#1}#2\end{array}\right]}

\newcommand{\closure}{\operatorname{clos}}
\newcommand{\interior}{\operatorname{int}}

\renewcommand{\Re}{\operatorname{Re}}
\renewcommand{\Im}{\operatorname{Im}}

\newtheorem{thm}{Theorem}[section]
\newtheorem{prop}[thm]{Proposition}
\newtheorem{lem}[thm]{Lemma}
\newtheorem{cor}[thm]{Corollary}
\theoremstyle{definition}

\newtheorem{rem}[thm]{Remark}
\numberwithin{equation}{section}

\title{\texorpdfstring{Eigenvalues of the tetradiagonal Toeplitz matrices\\with diagonals 1, 0, 0, 1}{Eigenvalues of the tetradiagonal Toeplitz matrices with diagonals 1, 0, 0, 1}}
\author{Sergei M. Grudsky,
Rom\'{a}n Higuera-Garc\'{i}a,\\[1ex]
Egor A. Maximenko,
Fidel V\'{a}squez-Rojas}
\date{\today}

\begin{document}
\maketitle

\begin{abstract}
We perform a thorough analysis
of the eigenvalues
of tetradiagonal Toeplitz matrices of large order $n$
generated by the Laurent polynomial
$a(t)=t^2+t^{-1}$.
The spectra of these matrices
are invariant under $2\pi/3$-rotation.
They are contained in three segments of the complex plane
and asymptotically fill these segments
as $n$ tends to infinity.
We apply Widom's formula for the determinants and transform the characteristic equation into a convenient form that can be solved by the fixed point iteration method.
The main results are asymptotic formulas for the eigenvalues,
both close to the origin and far from the origin.
The main results are verified by numerical tests
for moderate values of $n$.

\medskip\noindent
\textbf{Keywords:}
banded Toeplitz matrices,
non-Hermitian matrices,
eigenvalues,
asymptotic expansion,
fixed point.

\medskip\noindent
\textbf{MSC (2020):}
15B05, 15A18, 41A60.

\end{abstract}

\medskip
\tableofcontents

\bigskip
\section{Introduction and main results}
\label{sec:intro}

\noindent
Toeplitz matrices have been an object of extensive studies
for more than one century
(see some references in \cite{BoettcherSilbermann1999,BoettcherGrudsky2005,DeiftItsKrasovsky2011,GaroniSerra2017}).
There are many results about the collective asymptotic behavior
of their eigenvalues, i.e., about the determinants
and the asymptotic distribution.
In the last years, the individual eigenvalues were studied for Hermitian Toeplitz matrices with smooth generating symbols having only two intervals of increasing and decreasing, and satisfying some additional conditions~\cite{BoettcherGrudskyMaksimenko2010,BogoyaBottcherGrudskyMaximenko2015,BogoyaGrudskyMaximenko2017,Rambour2026}.
The asymptotic behavior of the eigenvalues is also investigated for some families of non-Hermitian Toeplitz matrices,
when the eigenvalues are close to the values of the generating symbol
\cite{Widom1994,BatalshchikovGrudskyMalishevaMihalkovichRamirezStukopin2019,BogoyaBoettcherGrudsky2012}.

Much less is known about the
individual eigenvalues of Toeplitz matrices in the non-Hermitian case,
especially when the eigenvalues are far from the image of the generating symbol.
Given a Laurent polynomial
\begin{equation}\label{eq:Laurent}
b(t)=\sum_{j=-r}^s b_j t^j
\end{equation}
and number $n$ in
$\bN\eqdef\{1,2,\ldots\}$,
the \emph{Toeplitz matrix}
$T_n(b)$ is defined by
$T_n(b)=\bigl[b_{j-k}\bigr]_{j,k=1}^n$,
where $b_q=0$ for
$q$ in $\bZ\setminus\{-r,\ldots,s\}$.
The function $b$ is called the \emph{generating symbol}
(or simply the \emph{symbol}) of the sequence of matrices $T_n(b)$.

The first steps in the study
of banded Toeplitz determinants
were made by
Widom~\cite{Widom1958},
Baxter and Schmidt~\cite{BaxterSchmidt1961}.
Schmidt and Spitzer~\cite{SchmidtSpitzer1960}
proved that the spectra $\spectrum T_n(b)$
converge (in the sense of the Hausdorff metric)
to a certain compact subset $\La(b)$ of the complex plane.
Hirschman~\cite{Hirschman1967} and Widom~\cite{Widom1990,Widom1994}
made some improvements of the result by Schmidt and Spitzer.
In particular, Hirschman computed the limiting distribution of the eigenvalues.
Ullman~\cite{Ullman1967} proved that $\La(b)$ is connected.
These results are also explained in~\cite[Chapter~11]{BoettcherGrudsky2005}.
Duits and Kuijlaars~\cite{DuitsKuijlaars2008} showed that the limiting measure of the eigenvalues can be described as the solution of an optimization problem: figuratively speaking, the eigenvalues asymptotically behave like repulsing particles on the set $\La(b)$.
B\"{o}ttcher, Gasca, Grudsky, and Kozak~\cite{BGGK2021}
explained a convenient method to solve the Schmidt--Spitzer equation which defines $\La(b)$.
They also analized some properties of $\La(b)$, especially in the tetradiagonal case ($r=1$, $s=2$).
In that case,
Widom's formula~\cite{Widom1958}
for the characteristic polynomials of $T_n(b)$ contains only three terms~\eqref{eq:Widom_det},
written in terms of the roots $z_1,z_2,z_3$ of $\la-b$.

Recently, Bogoya, Gasca, and Grudsky~\cite{BogoyaGascaGrudsky2024,BogoyaGascaGrudsky2025} considered tetradiagonal Toeplitz matrices $T_n(b)$
under some additional assumptions that guarantee
that the limiting set
$\La(b)$
is an analytic arc in the complex plane.
The thorough analysis in~\cite{BogoyaGascaGrudsky2024,BogoyaGascaGrudsky2025} showed that the third term in Widom's formula~\eqref{eq:Widom_det}
is exponentially small comparing to the first two terms.
After omitting that term and estimating the corresponding error, the authors of~\cite{BogoyaGascaGrudsky2024,BogoyaGascaGrudsky2025}
obtained an approximate characteristic equation,
somewhat similar to the equation in~\cite{BoettcherGrudskyMaksimenko2010}, and found an asymptotic expansion for the eigenvalues, as $n\to\infty$.

We also mention an investigation on another matrix family where the results were similar to some results of the present paper.
Da Fonseca and Veerman~\cite{FonsecaVeerman2009} studied the localization of the eigenvalues of a family of non-Hermitian tridiagonal almost-Toeplitz matrices.
Using an appropriate change of variables,
they transformed the characteristic equation to a trigonometric equation
(that can be solved numerically),
found the localization of the eigenvalues,
and computed the corresponding eigenvectors.

In this paper, we consider another family of tetradiagonal Toeplitz matrices.
Namely, we study Toeplitz matrices $T_n(a)$ generated by the Laurent polynomial
\begin{equation}
\label{eq:a}
a(t)\eqdef t^{-1}+t^2.
\end{equation}
The matrices $T_n(a)$ are tetradiagonal;
that is,
the width of the band is $s+r+1=4$. For instance,
\[
T_6(a)=
\matr{cccccc}{
0 & 1 & 0 & 0 & 0 & 0 \\
0 & 0 & 1 & 0 & 0 & 0 \\
1 & 0 & 0 & 1 & 0 & 0 \\
0 & 1 & 0 & 0 & 1 & 0 \\
0 & 0 & 1 & 0 & 0 & 1 \\
0 & 0 & 0 & 1 & 0 & 0
}.
\]
This family of Toeplitz matrices is very particular, but our results for this family (listed below) are quite complete and can serve as a model case for more general investigations.
In Remark~\ref{rem:comparison}, we emphasize the differences between this paper and~\cite{BoettcherGrudskyMaksimenko2010,BogoyaGascaGrudsky2024,BogoyaGascaGrudsky2025}.

Our goal is to study the behavior
of the eigenvalues of $T_n(a)$,
especially as $n$ tends to $\infty$.
This matrix family is a particular case of an example mentioned by Schmidt and Spitzer in~\cite[Section~7]{SchmidtSpitzer1960},
and the limiting set $\La(a)$
was already found there.
It is the following union of three segments centered and joined at the origin:
\begin{equation}
\label{eq:a_LimitingSet}
\La(a)
= [0,L]
\cup \left(\eps_3 [0,L]\right)
\cup \left(\eps_3^2 [0,L]\right).
\end{equation}
Here $\eps_3$ is the root of unity $\enumber^{2\pi\imunit/3}$ and
\begin{equation}
\label{eq:L_def}
L \eqdef \frac{3}{\sqrt[3]{4}},
\qquad
L \approx 1.88988.
\end{equation}
In fact, spectra of $T_n(a)$ are \emph{contained} in $\La(a)$ for every $n$,
see Figure~\ref{fig:symbol_and_eigenvalues_128}.

\begin{figure}[hbt]
\centering
\includegraphics{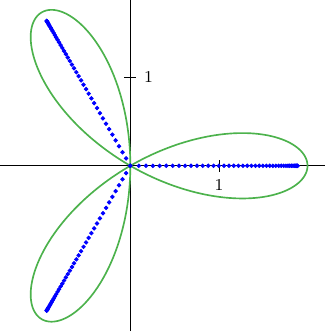}
\caption{The curve $\{a(t)\colon\ |t|=1\}$
and the eigenvalues of $T_{128}(a)$.
\label{fig:symbol_and_eigenvalues_128}}
\end{figure}

In this paper,
we always assume that $n\in\bN$ and $n\ge 3$.
In the division of $n$ by $3$,
we denote by $m$ and $r$ the quotient and remainder, respectively:
$m\eqdef\lfloor n/3\rfloor$,
$r\eqdef\remainder(n,3)$.
Equivalently,
\[
n = 3m+r,\qquad
m\in\bN,\qquad r\in\{0,1,2\}.
\]
For the sake of brevity, we suppress the dependence of $m$ and $r$ on $n$.

We start with two facts about 
localization and asymptotic distribution of the eigenvalues:
Theorem~\ref{thm:localization_of_spectrum} is known;
Theorem~\ref{thm:asymptotic_distribution} is known in other equivalent forms.
Theorem~\ref{thm:localization_of_spectrum}
is a particular case of a result by~McMillen~\cite{McMillen2009}
about double band matrices,
not necessarily Toeplitz.
Theorem~\ref{thm:localization_of_spectrum} also follows from a theorem by Eiermann and Varga about the zeros of some families of Faber polynomials~\cite[Theorem~3.2]{EiermannVarga1993}.
Already Schmidt and Spitzer~\cite{SchmidtSpitzer1960}
mentioned that the spectra of $T_n(a)$ are contained in $\La(a)$, without a proof.

\begin{thm}[localization and symmetry of the spectrum]
\label{thm:localization_of_spectrum}
Let $n\in\bN$, $n\ge3$.
Then, the spectrum of $T_n(a)$ is contained in $\La(a)$
and is invariant under $2\pi/3$-rotations around the origin.
More precisely, each one of the three open intervals
$\eps_3^p (0,L)$, where $p\in\{0,1,2\}$,
contains exactly $m$ simple eigenvalues.
If $r\ne0$,
then $0$ is an eigenvalue
of algebraic multiplicity $r$.
\end{thm}

In this paper, we obtain~Theorem~\ref{thm:localization_of_spectrum} as a simple by-product of more precise results on the approximation of the eigenvalues.
We also notice that for $r\ne 0$,
the geometric multiplicity of the eigenvalue $0$ is $1$
(Proposition~\ref{prop:origin_as_an_eigenvalue_of_Ta}).

We denote by $\la_{n,1},\ldots,\la_{n,m}$ the eigenvalues of $T_n(a)$ belonging to $(0,L)$ and written in ascending order:
\[
0<\la_{n,1}<\ldots<\la_{n,m}<L.
\]
Define
$F\colon[0,L]\to\left[0,\frac{1}{3}\right]$,
\begin{equation}
\label{eq:F_def}
F(\la)\eqdef
\frac{1}{3}-\frac{1}{\pi}
\arctan\left(\sqrt{3}\tanh\left(
\frac{1}{3}\arcosh\left(\frac{L}{\la}\right)^{3/2}
\right)\right).
\end{equation}
For $\la=0$, we put $F(0)\eqdef 0$.
Then, $F$ is continuous and strictly increasing on $[0,L]$.

Let $G\colon\left[0,\frac{1}{3}\right]\to[0,L]$
be the inverse function to $F$.
As we prove in Section~\ref{sec:roots_of_lambda_minus_a},
$G$ has the following explicit form:
\begin{equation}
\label{eq:G_explicit}
G(\al)
=
\frac{L \left(1-\frac{1}{3}\tan^2\left(\frac{\pi}{3}-\pi\al\right)\right)}%
{\left(1+\tan^2\left(\frac{\pi}{3}-\pi\al\right)\right)^{2/3}}.
\end{equation}
Figure~\ref{fig:F_and_G} shows the graphs of $F$ and $G$.
In these plots and in most figures in this paper,
the scales of the axes are not equal.

\begin{figure}[htb]
\centering
\includegraphics{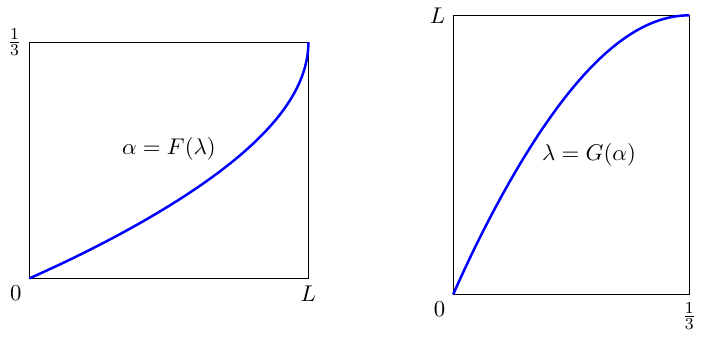}
\caption{Plots of $F$ and $G$.
\label{fig:F_and_G}}
\end{figure}

Given $n$ in $\bN$ with $n\ge 3$,
let
$H_n\colon\bR\to[0,1]$
be the \emph{empirical cumulative distribution function} of the strictly positive eigenvalues of $T_n(a)$:
\begin{equation}
\label{eq:empirical_distribution_function}
H_n(x)
\eqdef
\frac{\#\bigl\{j\in\{1,\ldots,m\}\colon\
\la_{n,j}\le x\bigr\}}{n}.
\end{equation}
The following Theorem~\ref{thm:asymptotic_distribution}
can be deduced from the article by
Coussement, Coussement, and Van Assche~\cite[Theorem~2.8]{CoussementCoussementAssche2008}
where the limiting eigenvalue distribution was calculated for
tetradiagonal Toeplitz matrices
with generating symbols of the form
$b(t)=t^{-1}+3\beta+3\beta t+t^2$.
In~\cite{CoussementCoussementAssche2008},
the density of the distribution is calculated in terms of some cubic and quadratic roots.
We give a more explicit expression~\eqref{eq:F_def} for the distribution function,
at the cost of using inverse hyperbolic functions.
In fact, Theorem~\ref{thm:asymptotic_distribution}
can be deduced from more general results
on the asymptotic distribution of the zeros of Faber polynomials;
see
He and Saff~\cite{HeSaff1994},
Aptekarev, Kalyagin, and Saff~\cite{AptekarevKalyaginSaff2009}.

\begin{thm}[the asymptotic distribution of the eigenvalues]
\label{thm:asymptotic_distribution}
For every $x$ in $[0,L]$,
\begin{equation}
\label{eq:limiting_distribution}
\lim_{n\to\infty}H_n(x)
=F(x).
\end{equation}
\end{thm}

Figure~\ref{fig:F_and_H_64}
shows the plots of $F$ and $H_{64}$.

\begin{figure}[htb]
\centering
\includegraphics{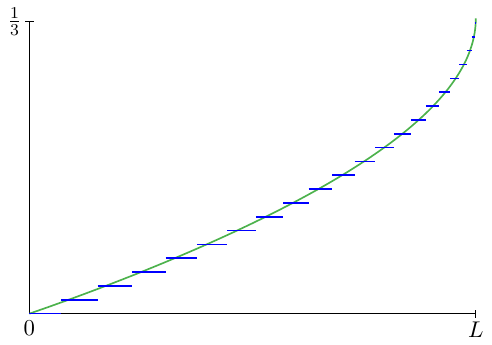}
\caption{Plots of $F$ and $H_{64}$.
\label{fig:F_and_H_64}}
\end{figure}

For every $j$ in $\{1,\ldots,m\}$, we put
$\al_{n,j}\eqdef F(\la_{n,j})$.
Equivalently, $\la_{n,j}=G(\al_{n,j})$.
Theorem~\ref{thm:asymptotic_distribution} implies that
the numbers $\al_{n,j}$ are uniformly distributed on $[0,1/3]$:
\begin{equation}
\label{eq:limiting_distribution_alpha}
\lim_{n\to\infty}
\frac{\#\bigl\{j\in\{1,\ldots,m\}\colon\
\al_{n,j}\le y\bigr\}}{n}
=y
\qquad \left(0\le y\le \frac{1}{3}\right).
\end{equation}
Therefore, $\al=F(\la)$ is the \emph{natural change of variables} in the characteristic equation.

The first main idea of this paper
is to make the trigonometric change of variable $\la=F(\al)$
in Widom's formula
for $\det(\la I_n-T_n(a))$;
thereby, we transform the characteristic equation from the algebraic form
to a more convenient trigonometric-exponential form (Theorem~\ref{thm:char_equation_as_fixed_point}).
For this purpose,
we need some auxiliary objects.
We define real-valued functions $\tht$, $p$, $q$ on $[0,1/3]$ by
\begin{align}
\label{eq:def_theta}
\tht(\al)
&\eqdef
\frac{1}{6}
-\frac{\al}{2}
-\frac{1}{\pi}
\arctan\frac{\tan\left(\frac{\pi}{3}-\pi\al\right)}{3},
\\
\label{eq:def_p}
p(\al)
&\eqdef
\frac{\tan\left(\frac{\pi}{3}-\pi\al\right)}%
{\sqrt{9+\tan^2\left(\frac{\pi}{3}-\pi\al\right)}},
\\
\label{eq:def_q}
q(\al)
&\eqdef
\frac{1}{2\cos\left(\frac{\pi}{3}-\pi\al\right)}.
\end{align}
The plots of these functions
are shown on Figure~\ref{fig:theta_p_q}.

\begin{figure}[htb]
\centering
\includegraphics{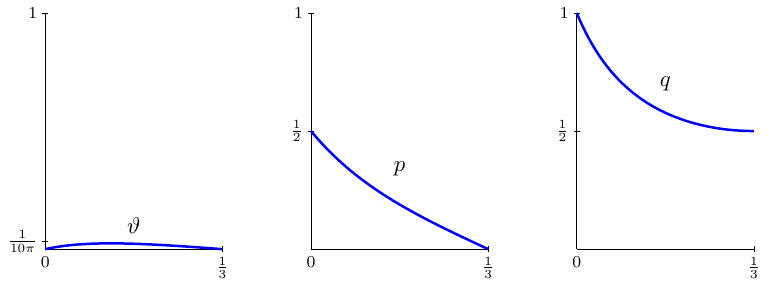}
\caption{Plots of $\tht$, $p$, and $q$.
\label{fig:theta_p_q}}
\end{figure}

Furthermore, for every $j$ in $\{1,\ldots,m\}$,
we put
\begin{equation}
\label{eq:initial_approximation}
\si_{n,j}
\eqdef
\frac{j-\frac{1}{2}+\frac{r}{3}}{n+\frac{3}{2}}
\end{equation}
and define
$\expterm_{n,j}\colon[0,1/3]\to\bR$
and $f_{n,j}\colon[0,1/3]\to[0,1/3]$ by
\begin{align}
\label{eq:expterm_def}
\expterm_{n,j}(\al)
&\eqdef
\frac{(-1)^{r+j+1}}{\bigl(n+\frac{3}{2}\bigr)\pi}
\arcsin\bigl(p(\al)q(\al)^{n+1}\bigr),
\\[1ex]
\label{eq:fnj_def}
f_{n,j}(\al)
&\eqdef
\si_{n,j}
-\frac{\tht(\al)}{n+\frac{3}{2}}
+\expterm_{n,j}(\al).
\end{align}
It is easy to see that the values of $f_{n,j}$ indeed belong to $[0,1/3]$.

\begin{thm}[a convenient form of the characteristic equation]
\label{thm:char_equation_as_fixed_point}
For every $n\ge 3$ and $j$ in $\{1,\ldots,m\}$,
the function $f_{n,j}$ is a contraction on
$\left[\frac{1}{3(n+1)},\frac{1}{3}\right]$,
and $\al_{n,j}$ is the unique solution of the equation
\begin{equation}
\label{eq:maineq}
\al=f_{n,j}(\al).
\end{equation}
\end{thm}

Theorem~\ref{thm:char_equation_as_fixed_point}
implies that the solution $\al_{n,j}$ of~\eqref{eq:maineq}
can be computed with any desired precision
by the simple iteration algorithm.
In Proposition~\ref{prop:equation_with_fnj_and_inital_approximation},
we show that $\al_{n,j}$ belongs to a certain small interval around $\si_{n,j}$.

The second main idea of this paper is to derive asymptotic expansions of
$\la_{n,j}$, as $n\to\infty$,
starting with equation~\eqref{eq:maineq}.

Because of the exponential factor $q^{n+1}$ in the right-hand side of~\eqref{eq:expterm_def},
if $n$ is large enough
and $j$ is not too small,
then $\expterm_{n,j}$ is very small.
Since $\tht$ is a smooth function,
we easily get an asymptotic expansion of the eigenvalues.

We define $A_0, A_1, A_2\colon[0,1/3]\to\bR$ by
\begin{equation}
\label{eq:A0A1A2_def}
A_0\eqdef G,\qquad
A_1\eqdef-G'\tht,\qquad
A_2\eqdef G'\tht\tht'
+\frac{1}{2}G''\tht^2.
\end{equation}
Explicit expressions for $\tht'$,
$G'$, and $G''$
are given in~\eqref{eq:tht_D1}
and~\eqref{eq:G_derivatives}.

\begin{thm}[asymptotic formula for the eigenvalues not too close to the origin]
\label{thm:lambda_asympt_not_too_close_to_the_origin}
There exists $C_1>0$ such that for every $n$ large enough and every $j$ in $\{\lfloor \log n\rfloor,\ldots,m\}$,
\begin{equation}
\label{eq:lambda_asympt_not_too_close_to_the_origin}
\la_{n,j}
=A_0(\si_{n,j})
+\frac{A_1(\si_{n,j})}{n+\frac{3}{2}}
+\frac{A_2(\si_{n,j})}{\left(n+\frac{3}{2}\right)^2}
+R_{1,n,j},
\end{equation}
where
\begin{equation}
|R_{1,n,j}|\le \frac{C_1}{\left(n+\frac{3}{2}\right)^3}.
\end{equation}
\end{thm}

In other words, the residue term in~\eqref{eq:lambda_asympt_not_too_close_to_the_origin}
can be written as $O(n^{-3})$,
and the upper bound of this residue is uniform with respect to $j$, where $\log n\le j\le m$.

For small values of $j$,
the term $\expterm_{n,j}$ cannot be dropped,
and the corresponding asymptotic formula for $\la_{n,j}$ is different.
To present it, we need the following auxiliary functions and numbers.
For every $r$ in $\{0,1,2\}$
and every $j$ in $\bN$,
we define
$\om_{r,j,0}$,
$\om_{r,j,1}$,
$\om_{r,j,2}$
from $[0,+\infty)$ to $\bR$ by
\begin{align}
\label{eq:om0_def}
\om_{r,j,0}(\ga)
&\eqdef 
\ga
-j+\frac{1}{2}-\frac{r}{3}
+\frac{(-1)^{r+j}}{\pi}
\arcsin\left(\frac{1}{2}\enumber^{-\sqrt{3}\pi\,\ga}\right),
\\[1ex]
\label{eq:om1_def}
\om_{r,j,1}(\ga)
&\eqdef
\frac{2(-1)^{r+j}\pi\ga^2}{(4 \enumber^{2\sqrt{3}\pi\ga }-1)^{1/2}},
\\[1ex]
\label{eq:om2_def}
\om_{r,j,2}(\ga)
&\eqdef
-\frac{2\sqrt{3}}{3}\pi\ga^2
+(-1)^{r+j}
\left(
\frac{2(-\pi\ga^2 -\frac{2\sqrt{3}}{3}\pi^2\ga^3 + \pi^3\ga^4)}{(4 \enumber^{2\sqrt{3}\pi\ga }-1)^{1/2}}
+
\frac{2\pi^3\ga^4}{(4\enumber^{2\sqrt{3}\pi\ga} -1)^{3/2}}
\right).
\end{align}
For each $r$ in $\{0,1,2\}$ and each $j$ in $\bN$, we define $\ph_{r,j}$
as the unique number $\ga$
in $[0,+\infty)$
satisfying the equation
$\om_{r,j,0}(\ga)=0$.
The last equation can also
be written in the form
\begin{equation}
\label{eq:ph_equation}
\ga
=
j-\frac{1}{2}+\frac{r}{3}
+\frac{(-1)^{r+j+1}}{\pi}
\arcsin\left(\frac{1}{2}
\enumber^{-\sqrt{3}\pi\,\ga}\right).
\end{equation}
Next, we define real numbers
$\psi_{r,j}$ and $\tau_{r,j}$ by
\begin{align}
\label{eq:psi_def}
\psi_{r,j}
&\eqdef
-\frac{\om_{r,j,1}(\phi_{r,j})}{\om_{r,j,0}'(\phi_{r,j})},
\qquad
\text{i.e.,}
\qquad
\psi_{r,j}
=
\frac{2\pi\ph_{r,j}^2}{\sqrt{3}+(-1)^{r+j+1}
\sqrt{4\enumber^{2\sqrt{3}\,\pi\,\ph_{r,j}}-1}},
\\[1ex]
\label{eq:tau_def}
\tau_{r,j}
&\eqdef
-\frac{1}{\om_{r,j,0}'(\phi_{r,j})}
\left(\frac{\om_{r,j,0}''(\phi_{r,j})}{2}\psi_{r,j}^2
+\om_{r,j,1}'(\phi_{r,j})\psi_{r,j}
+\om_{r,j,2}(\phi_{r,j})
\right),
\end{align}
The first derivatives of $\om_{r,j,0}$ and $\om_{r,j,1}$ can be computed explicitly, see~\eqref{eq:om_0_D1}--\eqref{eq:om_1_D1}.

The following theorem is the hardest result of this paper.

\begin{thm}[asymptotic formula for the eigenvalues close to the origin]
\label{thm:lambda_asympt_close_to_the_origin}
There exists $C_2>0$
such that for every $n$ large enough
and every
$j$ in $\{1,\ldots,\lfloor \log n\rfloor\}$,
\begin{equation}
\label{eq:asympt_near_zero}
\la_{n,j}
=G\left(\frac{\ph_{r,j}}{n+2}
+\frac{\psi_{r,j}}{(n+2)^2}
+\frac{\tau_{r,j}}{(n+2)^3}
\right)
+R_{2,n,j},
\end{equation}
where
$\ph_{r,j}$ is the solution of the transcendental equation~\eqref{eq:ph_equation},
$\psi_{r,j}$ and $\tau_{r,j}$
are given by~\eqref{eq:psi_def} and~\eqref{eq:tau_def},
respectively,
and
$|R_{2,n,j}|\le \frac{C_2 j^3}{(n+2)^4}$.
\end{thm}

Since $j$ in Theorem~\ref{thm:lambda_asympt_close_to_the_origin} is bounded by $\log n$,
the upper estimate for $|R_{2,n,j}|$
can also be written less precisely as
$O((\log n)^3/n^4)$.

For every $r$ and $j$,
the transcendental
equation~\eqref{eq:ph_equation}
can be solved by the simple iteration method
(see Lemma~\ref{lem:transcendental_equation_contractive}).
The number of such equations
is small in comparison to $n$.
For example, since
$\lfloor\log(16384)\rfloor=9$
and $r$ takes values in $\{0,1,2\}$,
it is sufficient to solve only $27$ equations of the form~\eqref{eq:ph_equation}
to apply the asymptotic formula~\eqref{eq:asympt_near_zero}
for all $n$ with $n\le 16384$.

Using the denominator $n+2$ in Theorem~\ref{thm:lambda_asympt_close_to_the_origin}
(instead of $n$, $n+1$, or $n+\frac{3}{2}$)
simplifies the formula for
$\om_{r,j,1}$
and slightly improves
the approximation of $\la_{n,j}$
by $G(\phi_{r,j}/(n+2))$.
Nevertheless, this choice does not change the asymptotic order of the residue term
in~\eqref{eq:asympt_near_zero}.

Since $G(0)=0$,
Theorem~\ref{thm:lambda_asympt_close_to_the_origin} yields the following rough approximation for the first positive eigenvalue ($j=1$):
\begin{equation}
\label{eq:first_eigenvalue_rough_approximation}
\la_{n,1} = \frac{G'(0)\ph_{r,1}}{n+2}
+ O\left(\frac{1}{(n+2)^2}\right).
\end{equation}
The constant coefficients in~\eqref{eq:first_eigenvalue_rough_approximation}
have the following numerical values:
\[
G'(0)=2\pi\sqrt{3}\approx 10.88,\qquad
\ph_{0,1}\approx 0.5099,\qquad
\ph_{1,1}\approx 0.8316,\qquad
\ph_{2,1}\approx 1.1669.
\]
Therefore, as $n$ changes, $\la_{n,1}$ has a ``ragged'' behavior depending on the remainder
$r=\remainder(n,3)$;
see Figure~\ref{fig:first_eigenvalue}.

\begin{figure}[htb]
\centering
\includegraphics{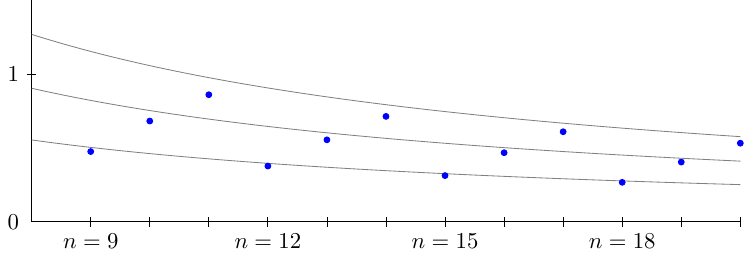}
\caption{First eigenvalues $\la_{n,1}$
for small values of $n$,
shown as blue points
$(n,\la_{n,1})$.
The gray hyperbolas are graphs
of the functions
$x \mapsto G'(0)\phi_{r,1}/x$,
$r\in\{0,1,2\}$.
\label{fig:first_eigenvalue}}
\end{figure}

\begin{rem}
The coefficients of the characteristic polynomials of $T_n(a)$ can be calculated explicitly
(see Proposition~\ref{prop:charpol_coefficients_explicit}).
Applying ideas from Van Assche~\cite{Assche2017},
it is possible to get Mehler--Heine type asymptotic formulas for these polynomials near the origin.
Following this way, one could express the limits
\[
\lim_{n\to\infty} n \la_{n,j}
=
\lim_{n\to\infty} (n+2)\la_{n,j}
=
G'(0)\phi_{r,j}
\]
in terms of the zeros of some 
hypergeometric functions.
This idea, suggested us by an anonymous referee, looks promising for further research.
Nevertheless, the results of~\cite{Assche2017}
only give one term of the asymptotic,
while Theorem~\ref{thm:lambda_asympt_close_to_the_origin} gives three terms.
\end{rem}

\begin{rem}[eigenvectors]
\label{rem:eigenvectors}
The eigenvectors banded Toeplitz matrices can be calculated by the formulas given by Trench~\cite{Trench1985eig}
or Maximenko and Moctezuma-Salazar~\cite{MaximenkoMoctezuma2017}.
Applying these general formulas to our case,
after some simplifications
we have found that if $j\in\{1,\ldots,m\}$ and $v_{n,j}\eqdef[v_{n,j,k}]_{k=1}^n$ is the vector with the components
\begin{equation}
\label{eq:main_eigenvector}
v_{n,j,k}
\eqdef
\frac{(-1)^{\left\lfloor\frac{k-1}{3}\right\rfloor}}%
{q(\al_{n,j})^{\frac{k-1}{3}}}
\left(
\cos\left(\left(k+\frac{1}{2}\right)\pi\al_{n,j}
+ \pi\tht(\al_{n,j})
- \frac{r_k\pi}{3}\right)
+ (-1)^{r_k}
p(\al_{n,j})
q(\al_{n,j})^k
\right),
\end{equation}
where $r_k=\remainder(k-1,3)$,
then $v_{n,j}$ is an eigenvector of $T_n(a)$
associated with the eigenvalue $\la_{n,j}$.
We are going to explain this formula and analyze the asymptotic behavior of $v_{n,j}$ in another paper.
\end{rem}

\begin{rem}[another numbering of the eigenvalues]
In~\eqref{eq:initial_approximation},
\eqref{eq:lambda_asympt_not_too_close_to_the_origin},
and~\eqref{eq:main_eigenvector},
we have an explicit dependence on $r\eqdef\remainder(n,3)$.
Numbering the eigenvalues in descending order
(that is, starting from the eigenvalues close to $L$)
and making the change of variable
$\al'=\frac{1}{3}-\al$ in the functions $\tht$, $p$, $q$, we could avoid the explicit dependence on $r$ and simplify some formulas;
see details in Remark~\ref{rem:decreasing_ordering}
and Corollary~\ref{cor:lambda_asympt_not_too_close_to_the_origin_reverse}.
Nevertheless, we prefer to pay more attention to the eigenvalues close to the origin.
This decision simplifies Theorems~\ref{thm:asymptotic_distribution} and~\ref{thm:lambda_asympt_close_to_the_origin}.
\end{rem}

\begin{rem}[comparison to \cite{BoettcherGrudskyMaksimenko2010,BogoyaGascaGrudsky2024,BogoyaGascaGrudsky2025}]
\label{rem:comparison}
The topic and the methods of this paper
have some similarities with \cite{BoettcherGrudskyMaksimenko2010} and~\cite{BogoyaGascaGrudsky2024,BogoyaGascaGrudsky2025}.
Nevertheless, here we have met and successfully solved the following new challenges.
\begin{itemize}
\item
Widom's formula~\eqref{eq:Widom_det}
applied to
$\la I_n-T_n(a)$
yields three terms.
For $\la$ close to zero,
unlike the situation in~\cite{BoettcherGrudskyMaksimenko2010,BogoyaGascaGrudsky2024,BogoyaGascaGrudsky2025}, 
the two bigger terms 
(which yield a trigonometric expression)
are not much bigger than the third term (which yields an exponential expression).
So, the third term plays an essential role and requires a special attention.

\item
The asymptotic expansion of the eigenvalues near the origin
(Theorem~\ref{thm:lambda_asympt_close_to_the_origin})
is much more complicated than the asymptotic expansions studied in \cite{BoettcherGrudskyMaksimenko2010,BogoyaGascaGrudsky2024,BogoyaGascaGrudsky2025}.
We present the main ideas in an abstract form (Section~\ref{sec:scheme_for_solving_equations}) and than apply them to our specific equations.

\item 
The behavior of the eigenvalues near the origin essentially depends on $r$,
where $r$ is the remainder of division of $n$ by $3$.
We have found a way to treat the corresponding three possible cases 
($r=0$, $r=1$, and $r=2$)
simultaneaosly in most situations,
including $r$ as a parameter in some equations
and avoiding separation into three different cases.

\item
Since the generating function $a$ in this paper has a very simple form,
it has been a natural task to obtain explicit expressions for the functions participating in the main theorems ($F$, $G$, $A_0$, $A_1$, $A_2$, etc.)
and good conditions for the analytic results;
for example, condition ``for all $n\ge 3$''
instead of ``for $n$ large enough''
in Theorem~\ref{thm:char_equation_as_fixed_point}.
We have made great efforts to obtain these explicit expressions and quite precise conditions.
In particular, we have used inverse hyperbolic functions to solve the cubic equation defining $z_1,z_2,z_3$ and get an explicit form of $F$.
\end{itemize}
\end{rem}

The rest of the paper contains the proofs of the theorems.
In Section~\ref{sec:symmetry_and_origin},
we show the symmetry of the spectrum under $\frac{2\pi}{3}$-rotations and prove that the origin is an eigenvalue, if $r\in\{1,2\}$.
In Section~\ref{sec:roots_of_lambda_minus_a},
we make several changes of variables
and compute the roots of the Laurent polynomial $\la-a$.
In Section~\ref{sec:char_pol},
we apply these changes of variables to the characteristic polynomial of $T_n(a)$.
In Section~\ref{sec:localization_of_spectrum},
we prove Theorem~\ref{thm:localization_of_spectrum}.
In Section~\ref{sec:analysis_of_transformed_char_equation_and_fixed_point}, we transform the characteristic equation to the form~\eqref{eq:maineq} and prove
Theorem~\ref{thm:char_equation_as_fixed_point}.
Section~\ref{sec:asymptotic_distribution}
contains a proof of Theorem~\ref{thm:asymptotic_distribution}.
In Section~\ref{sec:scheme_for_solving_equations},
we provide a general scheme
to represent solutions
of equations of the form
$\Omega(z(h),h)=0$
as asymptotic expansions by the small parameter $h$.
In Sections~\ref{sec:asympt_expansion_not_too_close_to_the_origin}
and~\ref{sec:close_to_the_origin},
we prove Theorems~\ref{thm:lambda_asympt_not_too_close_to_the_origin} and~\ref{thm:lambda_asympt_close_to_the_origin},
respectively.
Section~\ref{sec:numerical_tests} is devoted to numerical tests.

In Proposition~\ref{prop:trivial_generalization}, we provide a trivial generalization to the matrices of the form $T_n(b)$,
where
$b(t)=b_{-1}t^{-1}+b_0+b_2 t^2$
with $b_{-1},b_0,b_2\in\bC$,
$b_{-1}\ne0$, $b_2\ne0$.

\section{Symmetry, the origin, and characteristic polynomials}
\label{sec:symmetry_and_origin}

In this section, we recall explain the symmetry of the spectrum of $T_n(a)$,
give an explicit expression for the characteristic polynomials of $T_n(a)$,
and recall Widom's formula for the characteristic polynomials (or determinants).
Most results of this section are known,
see~\cite{Widom1958,SchmidtSpitzer1960,EiermannVarga1993,EkstromMeadon2023}.
We hope that Proposition~\ref{prop:origin_as_an_eigenvalue_of_Ta} is new.

\begin{rem}
\label{rem:Schmidt_and_Spitzer_trick}
In this remark,
we recall an idea by Schmidt and Spitzer\cite[page~19]{SchmidtSpitzer1960}.
For every $\rho$ in $\bC\setminus\{0\}$
and every Laurent polynomial $b$,
we denote by $b_\rho$
the Laurent polynomial defined by
\[
b_\rho(t)\eqdef b(\rho t).
\]
Let $\diag(d_1,d_2,\ldots,d_n)$ be the diagonal matrix with components $d_1,d_2,\ldots,d_n$ on the main diagonal.
Then,
\begin{equation}
\label{eq:SchmidtSpitzerIdentity}
T_n(b_{\rho})
=\diag(1,\rho,\ldots,\rho^{n-1})\,T_n(b)\,\diag(1,\rho^{-1},\ldots,\rho^{-n+1})
\end{equation}
and hence
\begin{equation}
\label{eq:spectrum_invariance}
\spectrum T_n(b_{\rho})
=\spectrum T_n(b).
\end{equation}
\end{rem}

\begin{prop}[symmetry of the spectrum of $T_n(a)$]
\label{prop:symmetry}
For every $n$, the spectrum of $T_n(a)$
is invariant under $2\pi/3$-rotations around the origin.
\end{prop}

\begin{proof}
Applying~\eqref{eq:spectrum_invariance}
to $b=a$ and $\rho=\eps_3=\enumber^{\frac{2\pi\imagunit}{3}}$
we see that
$\spectrum T_n(a)$ coincides with $\spectrum T_n(a_{\eps_3})$.
Furthermore, the polynomial $a$ has the following symmetry:
\begin{equation}\label{eq:symmetry_3}
a_{\eps_3}(t)
=a(\eps_3 t)
=\eps_3^{-1} t^{-1}+\eps_3^2 t^2
=\eps_3^{-1} t^{-1}+\eps_3^{-1}t^2
=\eps_3^{-1} a(t).
\end{equation}
Therefore,
$\spectrum T_n(a_{\eps_3})
=\eps_{3}^{-1}\spectrum T_n(a)$.
\end{proof}

\begin{prop}
\label{prop:rotation_eigenvectors}
Let $n\in\bN$, $n\ge3$,
$\la$ be an eigenvalue of $T_n(a)$
and $v=[v_k]_{k=1}^n$ be an associated eigenvector.
Then,
\[
w \eqdef \bigl[ \eps_3^{k-1} v_k \bigr]_{k=1}^n
\]
is an eigenvector associated with $\eps_3 \la$.
\end{prop}

\begin{proof}
We suppose that $T_n(a)v=\la v$
and $v\ne 0$.
Let
$D_n \eqdef \diag(1,\eps_3,\ldots,\eps_3^{n-1})$.
Then,
\[
T_n(a)
=\eps_3 T_n(\eps_3^{-1} a)
=\eps_3 T_n(a_{\eps_3})
=\eps_3 D_n T_n(a) D_n^{-1}.
\]
Since $w=D_n v$, we get
\[
T_n(a) w
= \eps_3 D_n T_n(a) D_n^{-1} D_n v
= \eps_3 D_n T_n(a) v
= \eps_3 D_n \la v
= \eps_3 \la\,w.
\qedhere
\]
\end{proof}

There are several equivalent formulas for the determinants of Toeplitz matrices in terms of the roots of the generating symbol,
see
\cite{Widom1958,BaxterSchmidt1961,Trench1985eig}
or~\cite[Chapter~2]{BoettcherGrudsky2005}.
Moreover, the determinants and minors of Toeplitz matrices can be written in 
terms of Schur or skew Schur polynomials,
see~\cite{Alexandersson2012,MaximenkoMoctezuma2017}.
For tetradiagonal Toeplitz matrices, if the generating symbol is of the form
\begin{equation}
\label{eq:tetradiagonal_general_symbol}
b(t)
=\sum_{k=-1}^2 b_k t^k
=t^{-1}(t-z_1)(t-z_2)(t-z_3),
\end{equation}
with $b_{-1}\ne0$, $b_2\ne0$,
and $z_1,z_2,z_3$ pairwise different nonzero numbers,
then Widom's formula simplifies to
\begin{equation}
\label{eq:Widom_det}
\det T_n(b)
=b_2^n
\left(\frac{(z_1 z_2)^{n+1}}{(z_1-z_3)(z_2-z_3)}
+\frac{(z_1z_3)^{n+1}}{(z_1-z_2)(z_3-z_2)}
+\frac{(z_2z_3)^{n+1}}{(z_2-z_1)(z_3-z_1)}
\right).
\end{equation}
Applying~\eqref{eq:Widom_det}
with $\la-b$ instead of $b$,
we get
\begin{equation}
\label{eq:Widom_det_lambda_minus_b}
\det T_n(\la-b)
=(-b_2)^n
\left(\frac{(z_1 z_2)^{n+1}}{(z_1-z_3)(z_2-z_3)}
+\frac{(z_1z_3)^{n+1}}{(z_1-z_2)(z_3-z_2)}
+\frac{(z_2z_3)^{n+1}}{(z_2-z_1)(z_3-z_1)}
\right),
\end{equation}
where $z_1,z_2,z_3$ are the roots of $\la-b$.

Instead of a general $b$,
we consider the particular Laurent polynomial
$a(t)=t^{-1}+t^2$.
We denote by $D_n$ the characteristic polynomial of $T_n(a)$:
\begin{equation}
\label{eq:charpol_notation}
D_n(\la)
\eqdef \det (\la I_n-T_n(a));
\quad\text{that is},\quad
D_n(\la)=\det T_n(\la-a).
\end{equation}
Here are the characteristic polynomials of $T_n(a)$ for small values of $n$:
\begin{equation}
\label{eq:char_pol_small_n}
D_3(\la)
=\la^3-1,
\qquad
D_4(\la)
=\la^4-2\la,
\qquad
D_5(\la)
=\la^5-3\la^2.
\end{equation}

The following well-known recursive relation provides a connection between $D_n$ and some families of Faber polynomials and multiple orthogonal polynomials,
see~\cite{EiermannVarga1993,AptekarevKalyaginSaff2009}.

\begin{prop}[recursive formula for the characteristic polynomials]
\label{prop:charpol_recursive}
For each $n\ge 3$ and each $\la$ in $\bC$,
\begin{equation}
\label{eq:charpol_recursive}
D_n(\la)
= \la D_{n-1}(\la)
- D_{n-3}(\la).
\end{equation}
\end{prop}

\begin{proof}
This recursive formula is easy to derive by expanding the determinant of $T_n(\la-a)$ by cofactors.
Another way to verify~\eqref{eq:charpol_recursive} is 
applying~\eqref{eq:Widom_det}
to $b(t)=-t^{-1}+\la t^0-t^2$
and taking into account that $z_1 z_2 z_3=-1$ and
$z_j^3-\la z_j+1=0$.
\end{proof}

\begin{prop}[explicit formula for the coefficients of the characteristic polynomial]
\label{prop:charpol_coefficients_explicit}
For each $n\ge 3$ with $n=3m+r$,
\begin{equation}
\label{eq:D_via_Q}
D_n(\la)
= \la^r Q_n(\la^3),
\end{equation}
where
\begin{equation}
\label{eq:Q_explicit}
Q_n(x)=\sum_{k=0}^m
(-1)^k
\binom{n-2k}{k} x^{m-k}.
\end{equation}
\end{prop}

\begin{proof}
Using 
\eqref{eq:char_pol_small_n},
\eqref{eq:charpol_recursive},
and mathematical induction,
it is easy to see that $D_n$
can be written in the form~\eqref{eq:D_via_Q},
and the polynomials $Q_n$ satisfy the following recursive formulas:
\begin{align*}
Q_{3m}(x)
&=x\,Q_{3m-1}(x)
-Q_{3m-3}(x),
\\
Q_{3m+1}(x)
&=Q_{3m}(x)-Q_{3m-2}(x),
\\
Q_{3m+2}(x)
&=Q_{3m+1}(x)-Q_{3m-1}(x).
\end{align*}
Applying mathematical induction,
it is easy to verify the explicit formula~\eqref{eq:Q_explicit}.
\end{proof}

Prof. Paul Barry explained to us by email
that the polynomial sequence defined by~\eqref{eq:char_pol_small_n}
and~\eqref{prop:charpol_recursive}
can be calculated by applying the theory of Riordan arrays,
see~\cite{Barry2022, Shapiro2022Riordan}.
We hope that connections between Toeplitz determinants and Riordan arrays can be useful in future works.
Nevertheless, in the other sections of this paper, we prefer using~\eqref{eq:Widom_det} rather than Proposition~\ref{prop:charpol_coefficients_explicit}.

\begin{prop}
\label{prop:origin_as_an_eigenvalue_of_Ta}
If $r=0$,
then $0$ is not an eigenvalue of $T_n(a)$.
If $r\in\{1,2\}$,
then $0$ is an eigenvalue of $T_n(a)$ of algebraic multiplicity $r$,
and the following vector $u$
is an associated eigenvector:
\begin{equation}
\label{eq:eigenvector_0}
u \eqdef \bigl[1,0,0,-1,0,0,1,0,0,-1,\ldots\bigr]^\top.
\end{equation}
If $r=2$, then the following vector $w$
is a generalized eigenvector of $T_n(a)$ associated with the eigenvalue $0$
and satisfying $T_n(a)w=u$:
\begin{equation}
\label{eq:generalized_eigenvector_0}
w=\bigl[
0, 1, 0, 0, -2, 0, 0, 3, 0, 0, -4, 0, 0, \ldots,
(-1)^m (m+1)\bigr]^\top.
\end{equation}
In particular, for $r=1$ or $r=2$,
the geometric multiplicity of the eigenvalue $0$ is $1$.
\end{prop}

\begin{proof}
1. Proposition~\ref{prop:charpol_coefficients_explicit} implies that if $r=0$, then $0$ is not an eigenvalue of $T_n(a)$,
and if $r\in\{1,2\}$,
then $0$ is an eigenvalue of $T_n(a)$ of algebraic multiplicity $r$.

2. Another way is to apply~\eqref{eq:Widom_det}.
The roots of the rational function $a(t)=t^2+t^{-1}$ are
$-1$, $\enumber^{\imunit\pi/3}$, $\enumber^{-\imunit\pi/3}$.
Therefore,
\[
\det(T_n(a))
=\frac{1}{3}\left(1+2\cos\frac{2n\pi}{3}\right).
\]
Therefore,
$\det(T_n(a))=0$
if and only if $r\ne 0$.
This reasoning is not sufficient to establish the algebraic multiplicity.

3. Suppose that $r\in\{1,2\}$.
We define $u=[u_j]_{j=1}^n$ by~\eqref{eq:eigenvector_0}.
A direct computation shows that $T_n(a)u=0$.

4. Finally, consider the case $r=2$.
Define $v=[v_j]_{j=1}^n\in\bC^n$ by 
\[
v_j \eqdef
\begin{cases}
(-1)^{k-1}\,k, & \text{if}\ j=2k+2, \\
0, & \text{otherwise}.
\end{cases}
\]
In other words, $v$ is given by~\eqref{eq:generalized_eigenvector_0}.
We see that $v$ is not a multiple of $u$
and $T_n(a)v=u$,
which means that $v$ is a generalized eigenvector.
\end{proof}

In Section~\ref{sec:localization_of_spectrum},
we will show that $T_n(a)$ has $3m$ nonzero eigenvalues.

Using the idea of Schmidt and Spitzer stated in Remark~\ref{rem:Schmidt_and_Spitzer_trick},
it is easy to pass from $T_n(a)$ to a slightly more general family of Toeplitz matrices.

\begin{prop}
\label{prop:trivial_generalization}
Let $b$ be a Laurent polynomial of the form
\[
b(t)=b_{-1}t^{-1}+b_0+b_2 t^2,
\]
where $b_{-1},b_0,b_2\in\bC$,
$b_{-1}\ne0$ and $b_2\ne0$.
Let $\rho\in\bC$ such that $\rho^3=b_2/b_{-1}$.
Then
\[
\spectrum T_n(b)
=b_0 + b_{-1}\rho \spectrum T_n(a).
\]
Moreover, if $\la\in\spectrum T_n(a)$,
$v=\bigl[v_k\bigr]_{k=1}^n\in\bC^n$,
and $T_n(a)v=\la v$,
then $T_n(b)w=(b_0+b_{-1}\rho\la)w$
for $w\eqdef\bigl[r^{k-1}v_k\bigr]_{k=1}^n$.
\end{prop}

\begin{proof}
Indeed,
\begin{align*}
b(t)
&= b_0 + b_{-1}
\left(t^{-1} + \frac{b_2}{b_{-1}}t^2\right)
=
b_0 + b_{-1} \rho
\left(\rho^{-1}t^{-1}
+ \rho^2 t^2\right)
\\
&=b_0 + b_{-1} \rho\,a(\rho t)
=b_0+b_{-1}\rho\,a_\rho(t).
\end{align*}
Therefore, by~\eqref{eq:SchmidtSpitzerIdentity},
the eigenvalues of $T_n(b)$ can be obtained from the eigenvalues of $T_n(a)$ by applying the function
$z \mapsto b_0 + b_{-1} \rho z$.
The statement about the eigenvectors
can be proved in the same way as Proposition~\ref{prop:rotation_eigenvectors},
but with
$D_n \eqdef \diag(1,\rho,\ldots,\rho^{n-1})$.
\end{proof}

\section{\texorpdfstring{Zeros of $\boldsymbol{\la-a}$ and changes of variables}{Zeros of lambda-a and changes of variables}}
\label{sec:roots_of_lambda_minus_a}

Using the symmetry of the spectra,
we restrict ourselves to the case
when $\la$ belongs to the real positive
part of the set $\SSS(a)$.
In order to apply Widom's formula~\eqref{eq:Widom_det}
and compute $D_n(\la)$,
we need formulas for the roots of the cubic polynomial
\[
t(a(t)-\la)=t^3-\la t+1.
\]
First, we recall that the roots of cubic equations can be expressed in terms of hyperbolic and inverse hyperbolic functions.
We denote by $\arcosh$ the inverse $\cosh$ function acting from $[1,+\infty)$ to $[0,+\infty)$.
Many authors prefer notation $\operatorname{arccosh}$
but the geometric sense of $\arcosh$ is an area, not an arc.

\begin{lem}
\label{lem:cubic_equation_0}
Let $A>1$ and
$\xi\eqdef\frac{1}{3}\arcosh(A)$.
Then, the cubic equation
\begin{equation}
\label{eq:cubic_equation_0}
4x^3-3x-A=0
\end{equation}
has three different roots:
\[
x_{1,2}
=-\frac{1}{2}\cosh(\xi)\pm\imunit\frac{\sqrt{3}}{2}\sinh(\xi),
\qquad
x_3=\cosh(\xi).
\]
\end{lem}

\begin{proof}
This fact is well known
(see, for example, \cite{Holmes2002}).
It can be verified by using the identities
\[
\cosh(3\xi)
=4\cosh^3(\xi)-3\cosh(\xi),\qquad
\sinh(3\xi)
=4\sinh^3(\xi)+3\sinh(\xi).
\]
The roots $x_1$ and $x_2$ can also be written as
$\cosh(\xi\pm 2\pi\imunit/3)$.
\end{proof}

Next, we pass from~\eqref{eq:cubic_equation_0}
to the equation $a(z)=\la$,
where $a(t)=t^2+t^{-1}$ is the generating symbol of our matrix family.
Recall that $L$ is defined by~\eqref{eq:L_def}.

\begin{prop}
\label{prop:cubic_equation_1}
Let $0<\la<L$ and
\begin{equation}
\label{eq:nu_from_la}
\nu=\frac{1}{3}\arcosh\left((L/\la)^{3/2}\right).
\end{equation}
Then, the cubic equation
\begin{equation}
\label{eq:cubic_equation_1}
z^3-\la z+1=0
\end{equation}
has three different roots:
\begin{equation}
\label{eq:z123_via_nu}
z_{1,2}
=\sqrt{\frac{\la}{3}}
\left(\cosh(\nu)
\pm\imunit\sqrt{3}\sinh(\nu)\right),
\qquad
z_3
=-2\sqrt{\frac{\la}{3}}\cosh(\nu).
\end{equation}
\end{prop}

\begin{proof}
We multiply both sides of the equation by $-\frac{3\sqrt{3}}{2\la\sqrt{\la}}$,
use the change of variable
$z = -2\sqrt{\frac{\la}{3}}\,x$,
and apply Lemma~\ref{lem:cubic_equation_0} with
\[
A
=\frac{3\sqrt{3}}{2\la\sqrt{\la}}
=\left(\frac{L}{\la}\right)^{3/2}.
\]
Notice that $z_1$ is obtained from $x_2$
and $z_2$ is obtained from $x_1$.
With this choice,
the imaginary part of $z_1$ is positive.
\end{proof}

After~\eqref{eq:nu_from_la},
we make another change of variables,
defining $\be$ as the principal argument of $z_1$:
\begin{equation}
\label{eq:be_via_nu}
\be
=\arg(z_1)
=\arctan\frac{\Im(z_1)}{\Re(z_1)}
=\arctan\left(\sqrt{3}\tanh(\nu)\right).
\end{equation}
Equivalently,
\begin{equation}
\label{eq:be_via_la}
\be
=\arctan
\left(\sqrt{3}\tanh\left(\frac{1}{3}\arcosh\bigl((L/\la)^{3/2}\bigr)\right)\right).
\end{equation}

\begin{prop}
\label{prop:roots_via_be}
Let $0<\la<L$ and $\be$ be defined by~\eqref{eq:be_via_la}.
Then, $\be\in\left(0,\frac{\pi}{3}\right)$ and
\begin{equation}\label{eq:la_via_be}
\la
=
\frac{L\,\left(1-\frac{1}{3}\tan^2\be\right)}{(1+\tan^2\be)^{2/3}}.
\end{equation}
Furthermore, the roots of the polynomial
$z^3-\la z+1$
can be written as
\begin{equation}
\label{eq:roots_via_be}
z_1 = \rho \enumber^{\imunit\be},\quad
z_2 = \rho \enumber^{-\imunit\be},\quad
z_3 = -\frac{1}{\rho^2} = - 2\rho \cos\be,
\end{equation}
where
\begin{equation}
\label{eq:rho_via_be}
\rho = \left(\frac{1}{2\cos\be} \right)^{1/3}.
\end{equation}
These numbers satisfy $|z_1|=|z_2|<1<|z_3|$.
\end{prop}

\begin{proof}
Let $\rho=|z_1|$. Then,
\begin{equation}
\label{eq:z12_via_rho_be}
z_1=\rho \enumber^{\imunit\be},\qquad
z_2=\rho \enumber^{-\imunit\be}.
\end{equation}
Vieta's formulas for the polynomial $z^3-\la z+1$
tell us that
\begin{gather}
z_1 z_2 z_3 = -1,
\label{eq:product_of_three}
\\
\label{eq:product_of_two}
z_1 z_2 + z_1 z_3 + z_2 z_3 = -\la,
\\
\label{eq:sum_of_roots}
z_1 + z_2 + z_3 = 0.
\end{gather}
From~\eqref{eq:z12_via_rho_be} and \eqref{eq:product_of_three}
we obtain
$z_3=-\frac{1}{\rho^2}$.
Now, by~\eqref{eq:sum_of_roots},
\begin{equation}
\label{eq:rho3_via_beta}
\rho^3 = \frac{1}{2\cos\be},
\end{equation}
which yields \eqref{eq:rho_via_be} and \eqref{eq:roots_via_be}.
Substituting \eqref{eq:roots_via_be} and~\eqref{eq:rho_via_be}
into \eqref{eq:product_of_two}
we obtain \eqref{eq:la_via_be}:
\begin{align*}
\la
&= \rho^2 (4\cos^2 \be - 1)
= \frac{1}{\sqrt[3]{4}} \cos^{4/3}\be \left(4-\frac{1}{\cos^2\be}\right)
\\
&= \frac{1}{\sqrt[3]{4}} \frac{3-\tan^2\be}{(1+\tan^2\be)^{2/3}}
= \frac{L\left(1-\frac{1}{3}\tan^2\be\right)}%
{(1+\tan^2\be)^{2/3}}.
\end{align*}
The function
\[
\be\mapsto
\frac{L\left(1-\frac{1}{3}\tan^2\be\right)}%
{(1+\tan^2\be)^{2/3}}
\]
strictly decreases on $[0,\pi/3]$
and takes values from $L$ to $0$.
Moreover, for $\be$ in $(0,\pi/3)$,
we obtain $\rho<1$ by~\eqref{eq:rho_via_be}
and $|z_1|<|z_3|$ by~\eqref{eq:roots_via_be}.
\end{proof}

Finally, we apply the change of variables
\begin{equation}
\label{eq:al_from_be}
\al=\frac{1}{\pi}\left(\frac{\pi}{3}-\be\right)
=\frac{1}{3}-\frac{\be}{\pi},
\qquad\text{i.e.},\qquad
\be=\frac{\pi}{3}-\pi\al.
\end{equation}
We denote by $F$ the correspondence
$\la\mapsto\al$ and by $G$ the inverse function.
The functions
$F$ and $G$ are given explicitly
by~\eqref{eq:F_def} and~\eqref{eq:G_explicit}.
This change of variables is justified by Theorem~\ref{thm:asymptotic_distribution}
which will be proved in Section~\ref{sec:asymptotic_distribution}.

\section{Characteristic polynomial after a change of variables}
\label{sec:char_pol}

In this section,
we represent the characteristic polynomial of $T_n(a)$ in a convenient form.
Recall that $D_n(\la)=\det(\la I_n-T_n(a))$;
see~\eqref{eq:charpol_notation}.
For the next proposition,
we need some auxiliar functions
defined on $(0,\pi/3)$:
\begin{align}
\label{eq:def_M_old}
\oldM(\be)
&\eqdef
\frac{2}{\sin\be\sqrt{9+\tan^2\be}},
\\[1ex]
\label{eq:def_E_old}
\oldE_n(\be)
&\eqdef \sin\bigl((n+1)\be + \oldtht(\be) \bigr)
+(-1)^n
\oldp(\be)
\oldq(\be)^{n+1},
\\[1ex]
\label{eq:def_theta_p_q_old}
\oldtht(\be)
&\eqdef \arctan\left (\frac{\tan\be}{3}\right),
\qquad
\oldp(\be)
\eqdef \frac{\tan\be}{\sqrt{9+\tan^2\be}},
\qquad
\oldq(\be) \eqdef \frac{1}{2\cos\be}.
\end{align}

\begin{prop}[the characteristic polynomial in terms of $\be$]
\label{prop:charpol_via_beta}
Let $0<\la<L$ and $\be$ be defined by
\eqref{eq:be_via_la}.
Then,
\begin{equation}
\label{eq:charpol_via_be}
D_n(\la)
=
\oldM(\be)\,
(2\cos\be)^{n/3}\,
\oldE_n(\be).
\end{equation}
\end{prop}

\begin{proof}
We are going to apply~\eqref{eq:Widom_det}
to $b=\la-a$
and rewrite it in terms of $\be$
using~\eqref{eq:roots_via_be}.
The difference $z_1-z_2$ is just $2\imunit \rho\sin\be$.
Consider the difference $z_1-z_3$:
\[
z_1-z_3
= \rho (2\cos\be + \enumber^{\imunit\be})
= \rho (3\cos\be+\imunit\sin\be).
\]
The argument of $z_1-z_3$
is $\oldtht(\be)$, defined by~\eqref{eq:def_theta_p_q_old},
while its absolute value is
\[
|z_1-z_3|
=\rho\,\cos\be\;\sqrt{9+\tan^2\be}.
\]
Here, we have used the fact
$\cos(\be)>0$ for $\be$ in $(0,\pi/3)$.
Applying the formula
$z_2-z_3=\overline{z_1-z_3}$,  
we obtain
\begin{equation}
\label{eq:diferences}
z_1-z_3
=\rho\,\cos\be\,\sqrt{9+\tan^2\be}\,
\enumber^{\imunit\oldtht(\be)},
\qquad
z_2-z_3
=\rho\,\cos\be\,\sqrt{9+\tan^2\be}\,
\enumber^{-\imunit\oldtht(\be)}.
\end{equation}
We denote by 
$S_n^{\{1,2\}}$,
$S_n^{\{1,3\}}$,
and
$S_n^{\{2,3\}}$
the terms in the right-hand
side of~\eqref{eq:Widom_det}.
Using the expressions above and~\eqref{eq:rho3_via_beta}, we get
\begin{align*}
S_n^{\{1,2\}}
&=
\frac{(z_1 z_2)^{n+1}}{(z_1-z_3)(z_2-z_3)}
= \frac{\rho^{2n}}%
{\cos^2\be\,(9+\tan^2\be)},
\\[1ex]
S_n^{\{1,3\}}
&=
\frac{(z_1 z_3)^{n+1}}{(z_1-z_2)(z_3-z_2)}
=
\frac{(-1)^{n+1} \rho^{-n}}%
{\sin\be\,\sqrt{9+\tan^2\be}}\,
\imunit\enumber^{\imunit\,((n+1)\be+\oldtht(\be))},
\\[1ex]
S_n^{\{2,3\}}
&=
\frac{(z_2 z_3)^{n+1}}{(z_2-z_1)(z_3-z_1)}
=\overline{S_n^{\{1,3\}}}
=
-\frac{(-1)^{n+1} \rho^{-n}}%
{\sin\be\,\sqrt{9+\tan^2\be}}\,
\imunit\enumber^{-\imunit\,((n+1)\be+\oldtht(\be))}.
\end{align*}
Thus, the characteristic polynomial of $T_n(a)$ is
\begin{align*}
D_n(\la)
&=
(-1)^n 
\Bigl(S_n^{\{1,2\}}+S_n^{\{1,3\}}+S_n^{\{2,3\}}\Bigr)
\\
&=\frac{2\rho^{-n}}%
{\sin\be\,\sqrt{9+\tan^2\be}}\,
\sin\Bigl((n+1)\be+\oldtht(\be)\Bigr)
+\frac{(-1)^n\rho^{2n}}{\cos^2\be (9+\tan^2\be)}
\\
&=\oldM(\be)\,(2\cos\be)^{n/3}
\left(\sin\Bigl((n+1)\be+\oldtht(\be)\Bigr)
+
\frac{(-1)^n\tan\be}{\sqrt{9+\tan^2\be}}
\left(\frac{1}{2\cos\be}\right)^{n+1}\right),
\end{align*}
which equals~\eqref{eq:charpol_via_be}.
\end{proof}

Now, we pass to the variable
$\al$ related to $\be$ by~\eqref{eq:al_from_be}.
We define real-valued functions $\tht$, $p$, and $q$
on $[0,1/3]$ by~\eqref{eq:def_theta}--\eqref{eq:def_q}.
After that, we define $E_n\colon[0,1/3]\to\bR$ by
\begin{equation}
\label{eq:E_def}
E_n(\al)
\eqdef
\cos\left(\left(n+\frac{3}{2}\right)\pi\al
+ \pi\tht(\al)
- \frac{r\pi}{3}\right)
+(-1)^r
p(\al)
q(\al)^{n+1}.
\end{equation}
We also define $M\colon[0,1/3]\to\bR$,
\begin{equation}
M(\al)
\eqdef
\frac{2}{\sin\left(\frac{\pi}{3}-\pi\al\right)
\sqrt{9+\tan^2\left(\frac{\pi}{3}-\pi\al\right)}}.
\end{equation}

\begin{prop}[the characteristic polynomial in terms of $\al=F(\la)$]
\label{prop:charpol_via_alpha}
Let $0<\la<L$ and $\al=F(\la)$.
Then,
\begin{equation}
\label{eq:charpol_via_al}
D_n(\la)
=(-1)^m M(\al)\,
\left(\frac{1}{q(\al)}\right)^{n/3}\,
E_n(\al).
\end{equation}
\end{prop}

\begin{proof}
This proposition follows from 
Proposition~\ref{prop:charpol_via_beta},
using the change of variable~\eqref{eq:al_from_be}.
Notice that
\begin{align*}
p(\al)
&=\oldp\left(\frac{\pi}{3}-\pi\al\right),
&
q(\al)
&=\oldq\left(\frac{\pi}{3}-\pi\al\right),
\\
M(\al)
&=
\oldM\left(\frac{\pi}{3}-\pi\al\right),
&
\tht(\al)
&=\frac{1}{6}-\frac{\al}{2}
-\frac{1}{\pi}
\oldtht\left(\frac{\pi}{3}-\pi\al\right).
\end{align*}
The expression $\oldE_n(\be)$ transforms as follows:
\begin{align*}
\oldE_n(\be)
&=\oldE_n\left(\frac{\pi}{3}-\pi\al\right)
=\sin\left((n+1)\left(\frac{\pi}{3}-\pi\al\right)
+\frac{\pi}{6}-\frac{\pi\al}{2}
-\pi\tht(\al)\right)
\\
&=\sin\left(m\pi+\frac{r\pi}{3}
-\left(n+\frac{3}{2}\right)\pi\al
+\frac{\pi}{2}
-\pi\tht(\al)\right)
=(-1)^m E_n(\al).
\end{align*}
Finally,
$(-1)^n=(-1)^{3m+r}=(-1)^m (-1)^r$.
\end{proof}

Proposition~\ref{prop:charpol_via_alpha} implies
that the equation $D_n(\la)=0$,
for $\la$ in $(0,L)$,
is equivalent to the equation
$E_n(\al)=0$, where $\al\in(0,1/3)$
and $\al=F(\la)$.

Let us mention some important properties of
$\tht$, $p$, and $q$.

\begin{prop}[properties of $\tht$] \mbox{}
\label{prop:theta}
\begin{enumerate}
\item $\tht$ is infinitely smooth on $[0,1/3]$.
\item $\tht(0)=0$, $\tht(1/3)=0$.
\item $\displaystyle\min_{0\le \al\le 1/3}\tht(\al)=0$.
\item $\displaystyle\max_{0\le \al \le 1/3}\tht(\al)
=\tht\left(\frac{1}{3}-\frac{1}{\pi}\arctan\frac{\sqrt{15}}{5}\right)
=\frac{1}{2\pi}\arctan\frac{\sqrt{15}}{5}
-\frac{1}{\pi}\arctan\frac{\sqrt{15}}{15}
\approx 0.0245$.
\item $\displaystyle\max_{0\le\al\le 1/3}|\tht'(\al)|=\tht'(0)=1/2$.
\end{enumerate}
\end{prop}

\begin{proof}
The derivative of $\tht$ is
\begin{equation}
\label{eq:tht_D1}
\tht'(\al)
=\frac{5}{2}
-
\frac{24}{9+\tan^2\left(\frac{\pi}{3}-\pi\al\right)}
=\frac{5}{2}
-\frac{24\cos^2\left(\frac{\pi}{3}-\pi\al\right)}{8\cos^2\left(\frac{\pi}{3}-\pi\al\right)+1}.
\end{equation}
Therefore, $\tht$ can be extended to an analytic function
in an open subset of $\bC$ containing  $[0,1/3]$.
Furthermore, $\tht$ strictly increases from $0$ to
$\frac{1}{3}-\frac{1}{\pi}\arctan\frac{\sqrt{15}}{5}$
and strictly decreases from this point to $1/3$.
\end{proof}

\begin{prop}[properties of $p$ and $q$]
\label{prop:p_q_properties}
Functions $p$ and $q$ are infinitely smooth on $[0,1/3]$
and strictly decreasing on $[0,1/3]$.
For every $\al$ in $[0,1/3]$, 
\[
0\le p(\al)\le\frac{1}{2},\qquad
\frac{1}{2}\le q(\al)\le 1,
\]
\[
|p'(\al)|\le\frac{\pi\sqrt{3}}{2},\qquad
|q'(\al)|\le \pi \sqrt{3}.
\]
\end{prop}

\begin{proof}
The statements are verified by direct computations.
\end{proof}

\section{Bisection method and spectrum localization}
\label{sec:localization_of_spectrum}

In this short section,
we show that Theorem~\ref{thm:localization_of_spectrum} can be easily proved by analyzing the characteristic equation in the form $E_n(\al)=0$,
see~\eqref{eq:E_def}.
For every $j$ in $\{0,1,\ldots,m\}$, let
\begin{equation}
\label{eq:xi_def}
\xi_{n,j} \eqdef
\frac{j+\frac{r}{3}}{n+\frac{3}{2}}.
\end{equation}

\begin{prop}
\label{prop:rough_localization_of_solutions}
For every $j$ in $\{1,\ldots,m\}$,
$E_n$ has at least one zero on
$(\xi_{n,j-1},\;\xi_{n,j})$.
\end{prop}

\begin{proof}
We represent $E_n(\al)$ as
$E_n(\al)=\cos(A_n(\al))+B_n(\al)$,
where
\[
A_n(\al)
\eqdef 
\left(n+\frac{3}{2}\right)\pi\al
+ \pi\tht(\al)
-\frac{r\pi}{3},\qquad
B_n(\al)
\eqdef
(-1)^r p(\al) q(\al)^{n+1}.
\]
By Proposition~\ref{prop:p_q_properties},
the term $B_n(\al)$ is absolutely bounded by $1/2$:
\[
|B_n(\al)| \le p(\al) \le \frac{1}{2}.
\]
On the other hand,
for each $j$ in $\{0,\ldots,m\}$,
\[
A_n(\xi_{n,j})
=j\pi+\pi\tht(\xi_{n,j}).
\]
Therefore, $A_n(\xi_{n,j})$ is close to $j\pi$:
\[
\left|A_n(\xi_{n,j})-j\pi\right|
=\bigl|\pi\tht(\xi_{n,j})\bigr|
<\frac{1}{10}.
\]
Since $\cos(j\pi)=(-1)^j$
and $\cos$ is Lipschitz continuous with coefficient $1$,
this implies that
\[
\left|
\cos(A_n(\xi_{n,j})) - (-1)^j
\right|
< \frac{1}{10}.
\]
\[
\left|
(-1)^j \cos(A_n(\xi_{n,j})) - 1
\right|
< \frac{1}{10}.
\]
Thus,
$\cos(A_n(\xi_{n,j}))$ is close to $(-1)^j$:
\[
(-1)^j \cos(A_n(\xi_{n,j}))
> \frac{9}{10}.
\]
We conclude that
$(-1)^j E_n(\xi_{n,j})>\frac{9}{10}-\frac{1}{2}=\frac{2}{5}$ and
\[
\sign(E_n(\xi_{n,j}))=(-1)^j.
\]
For every $j$ in $\{1,\ldots,m\}$,
$E_n$ has different signs at the points $\xi_{n,j-1}$ and $\xi_{n,j}$.
Moreover, $E_n$ is continuous.
By the intermediate value theorem,
$E_n$ has a zero on
$(\xi_{n,j-1},\xi_{n,j})$.
\end{proof}

The last lines of the proof of Proposition~\ref{prop:rough_localization_of_solutions} imply that the equation $E_n(\al)=0$ can be solved 
by the bisection method on
$(\xi_{n,j-1},\xi_{n,j})$.

Figures~\ref{fig:char_pol_vanishing_factor_20},
\ref{fig:char_pol_vanishing_factor_21},
and~\ref{fig:char_pol_vanishing_factor_22}
show the geometric sense of Proposition~\ref{prop:rough_localization_of_solutions}.

\begin{figure}[htb]
\centering
\includegraphics{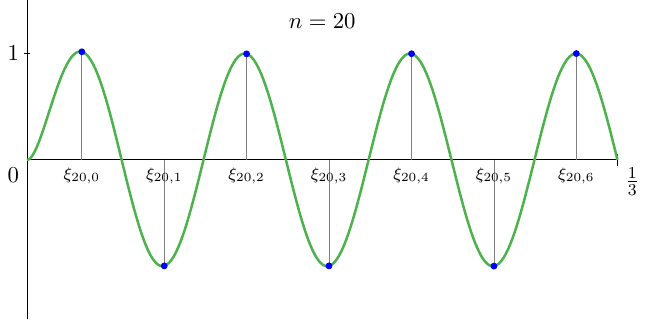}
\caption{Graph of $E_{20}$.
\label{fig:char_pol_vanishing_factor_20}}
\end{figure}

\begin{figure}[htb]
\centering
\includegraphics{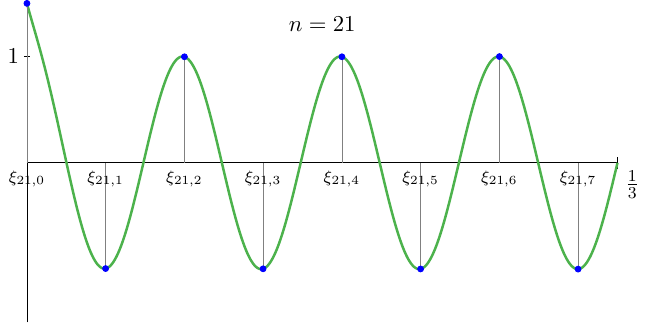}
\caption{Graph of $E_{21}$.
\label{fig:char_pol_vanishing_factor_21}}
\end{figure}

\begin{figure}[htb]
\centering
\includegraphics{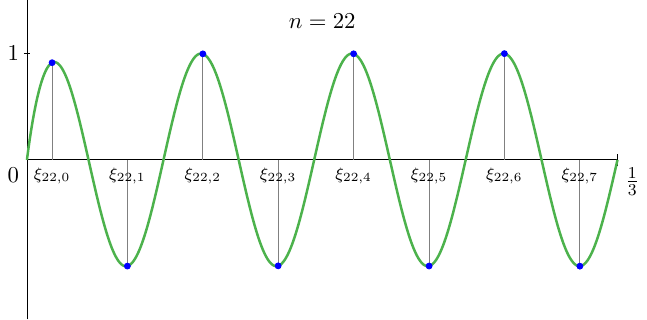}
\caption{Graph of $E_{22}$.
\label{fig:char_pol_vanishing_factor_22}}
\end{figure}

\begin{prop}
\label{prop:localization_of_strictly_positive_eigenvalues}
For every $j$ in $\{1,\ldots,m\}$,
$T_n(a)$ has an eigenvalue
in the interval
\[
\Bigl(G_n(\xi_{n,j-1}),
G_n(\xi_{n,j})\Bigr).
\]
\end{prop}

\begin{proof}
This is a simple consequence of Proposition~\ref{prop:rough_localization_of_solutions}.
Indeed, $G$ is a strictly increasing function,
and for every $\al$ in $(0,1/3)$,
the equation
$E_n(\al)=0$
is equivalent to the characteristic equation
$\det(T_n(\la-a))=0$
with $\la=G(\al)$.
\end{proof}

\begin{proof}[Proof of Theorem~\ref{thm:localization_of_spectrum}.]
By Proposition~\ref{prop:localization_of_strictly_positive_eigenvalues},
$T_n(a)$ has at least $m$ different eigenvalues on $(0,L)$.
Thus, by Proposition~\ref{prop:symmetry},
$T_n(a)$ has at least $3m$ eigenvalues different from the origin.
By Proposition~\ref{prop:origin_as_an_eigenvalue_of_Ta},
the origin is an eigenvalue of algebraic multiplicity at least $r$.
Since $T_n(a)$ has exactly $n$ eigenvalues, counting with algebraic multiplicities,
we conclude that we have found all of them;
that is, there are no other eigenvalues.
\end{proof}

\section{Justification of the fixed-point method}
\label{sec:analysis_of_transformed_char_equation_and_fixed_point}

In this section, we transform the characteristic equation to the form $\al=f_{n,j}(\al)$
and prove Theorem~\ref{thm:char_equation_as_fixed_point}.
We suppose that $n\in\bN$ and $n\ge 3$.

\begin{prop}
\label{prop:equation_with_fnj_and_inital_approximation}
For every $j$ in $\{1,\ldots,m\}$,
$\al_{n,j}$
is a fixed point of $f_{n,j}$.
Moreover,
\begin{equation}
\label{eq:al_nj_minus_si_nj_upper_bound}
|\al_{n,j}-\si_{n,j}|
<\frac{1}{5\left(n+\frac{3}{2}\right)}.
\end{equation}
Equivalently,
\begin{equation}
\label{eq:upper_and_lower_bounds_for_al}
\frac{j-\frac{7}{10}+\frac{r}{3}}{n+\frac{3}{2}}
<
\al_{n,j}
<
\frac{j-\frac{3}{10}+\frac{r}{3}}{n+\frac{3}{2}}.
\end{equation}
\end{prop}

\begin{proof}
By Proposition~\ref{prop:charpol_via_alpha},
the characteristic equation is transformed to the form $E_n(\al)=0$; that is,
\[
\sin\left(\left(n+\frac{3}{2}\right)\pi\al
+\pi\tht(\al)
+\frac{\pi}{2}
- \frac{r\pi}{3}\right)
=(-1)^{r+1}
p(\al)
q(\al)^{n+1}.
\]
It is well known that if $y\in[-1,1]$, then
\[
\sin(x)=y
\qquad\Longleftrightarrow\qquad
\exists j\in\bZ\qquad
x=j\pi+(-1)^j\arcsin(y).
\]
Applying this fact with
$y=(-1)^{r+1} p(\al)q(a)^{n+1}$,
we see that the equation $E_n(\al)=0$
is equivalent to the union (over $j$ in $\bZ$)
of the equations
\[
\left(n+\frac{3}{2}\right)\pi\al
+\pi\tht(\al)
+\frac{\pi}{2}
-\frac{r\pi}{3}
=j\pi+(-1)^{r+j+1}
\arcsin\left(p(\al)q(\al)^{n+1}\right),
\]
i.e., with notation~\eqref{eq:initial_approximation}
and~\eqref{eq:expterm_def},
\begin{equation}
\label{eq:transformed_equation}
\al
= 
\si_{n,j}
- \frac{\tht(\al)}{n+\frac{3}{2}}
+ \expterm_{n,j}(\al).
\end{equation}
Since $|\tht(\al)|<\frac{1}{10\pi}$
and
\[
\frac{1}{\pi}
\arcsin\bigl(p(\al)q(\al)^{n+1}\bigr)
<\frac{1}{\pi}\arcsin\frac{1}{2}
=\frac{1}{6},
\]
every solution of~\eqref{eq:transformed_equation} should satisfy
\[
|\al-\si_{n,j}|
<\frac{\frac{1}{10\pi}+\frac{1}{6}}{\left(n+\frac{3}{2}\right)}
<\frac{1}{5\left(n+\frac{3}{2}\right)}.
\]
Taking into account that
$0<\al<\frac{1}{3}$,
we easily conclude that
$j\in\{1,\ldots,m\}$.
Thereby, we can identify the solution of~\eqref{eq:transformed_equation}
with $\al_{n,j}$.

Finally, the inequalities~\eqref{eq:upper_and_lower_bounds_for_al} are equivalent
to~\eqref{eq:al_nj_minus_si_nj_upper_bound}
because
\begin{align*}
\si_{n,j}+\frac{1}{5\left(n+\frac{3}{2}\right)}
&=\frac{j-\frac{1}{2}+\frac{r}{3}}{n+\frac{3}{2}}
+\frac{1}{5\left(n+\frac{3}{2}\right)}
=
\frac{j-\frac{3}{10}+\frac{r}{3}}{n+\frac{3}{2}},
\\[0.5ex]
\si_{n,j}-\frac{1}{5\left(n+\frac{3}{2}\right)}
&= \frac{j-\frac{1}{2}+\frac{r}{3}}{n+\frac{3}{2}}
- \frac{1}{5\left(n+\frac{3}{2}\right)}
= \frac{j-\frac{7}{10}+\frac{r}{3}}{n+\frac{3}{2}}.
\qedhere
\end{align*}
\end{proof}

\begin{prop}[the initial approximation of the eigenvalues]
\label{prop:initial_approximation_of_the_eigenvalues}
Let $j\in\{1,\ldots,m\}$.
Then,
\begin{equation}
\label{eq:initial_approximation_of_the_eigenvalues}
\left|\la_{n,j}-G(\si_{n,j})\right|
\le \frac{2\sqrt{3}\,\pi}{5\left(n+\frac{3}{2}\right)}.
\end{equation}
\end{prop}

\begin{proof}
Explicit formulas for $G'$ and $G''$
(see~\eqref{eq:G_derivatives})
show that $G'$ is a positive decreasing function.
Therefore,
\[
\|G'\|_\infty = G'(0) = 2\sqrt{3}\,\pi.
\]
Now,~\eqref{eq:al_nj_minus_si_nj_upper_bound} and
the mean value theorem yield~\eqref{eq:initial_approximation_of_the_eigenvalues}.
\end{proof}

Figure~\ref{fig:G_and_initial_approximation} shows the approximation of $\la_{n,j}$ by $G(\si_{n,j})$
for $n=64$.

\begin{figure}[htb]
\centering
\includegraphics{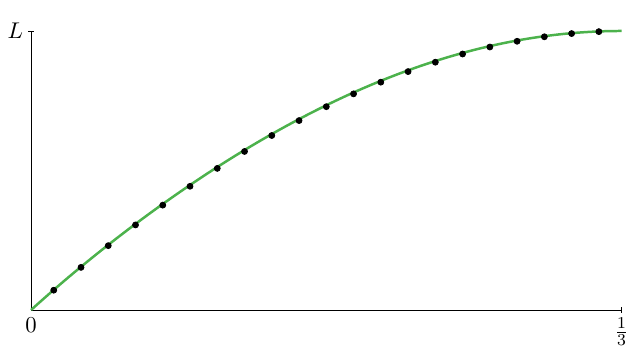}
\caption{The plot of $G$
and the points $(\si_{64,j},\la_{64,j})$,
$1\le j\le 21$.
\label{fig:G_and_initial_approximation}}
\end{figure}

\begin{prop}
\label{prop:fnj_not_contractive}
For every $j$ in $\{1,\ldots,m\}$,
$f_{n,j}$ is not contractive on $[0,1/3]$.
\end{prop}

\begin{proof}
A direct computation yields
\[
f_{n,j}'(0)
=\frac{\frac{1}{2}+(-1)^{j+r} (n+2)}{n+\frac{3}{2}}.
\]
If $j+r$ is odd, then $|f_{n,j}'(0)|=1$.
If $j+r$ is even, then
$|f_{n,j}'(0)|>1$.
In both cases, $f_{n,j}$ is not contractive on $[0,1/3]$.
\end{proof}

Inspired by
Propositions~\ref{prop:equation_with_fnj_and_inital_approximation}
and~\ref{prop:fnj_not_contractive},
we will exclude a neighborhood of $0$
and find a subinterval of $[0,1/3]$ where $f_{n,j}$ is a contraction.
Here, our main technical task
is to estimate $|\expterm_{n,j}'|$ from above.

\begin{lem}
\label{lem:log_q_denom_lower_bound}
Let $\al\in(0,1/12]$.
Then,
\begin{equation}
\label{eq:log_q_denom_lower_bound}
\log\frac{1}{q(\al)}
\ge 6\log(2)\,\al.
\end{equation}
\end{lem}

\begin{proof}
In this proof,
let $u(\la)$ be the left-hand side of~\eqref{eq:log_q_denom_lower_bound}:
\[
u(\al)
\eqdef
\log \frac{1}{q(\al)};
\qquad \text{that is,}\qquad
u(\al)
=\log\left(2\cos\left(\frac{\pi}{3}
-\pi\al\right)\right).
\]
It is easy to see that $u''<0$ on $[0,1/12]$.
Therefore, $u$ is strictly concave on $[0,1/12]$, and
\[
u(\al)
= u\left((1-12\al)\cdot 0+12\al\cdot\frac{1}{12}\right)
\ge (1-12\al)u(0)+12\al\,u(1/12)
= 6\log(2)\,\al.
\qedhere
\]
\end{proof}

\begin{lem}\label{lem:qn_upper_bound}
Let $n\ge 3$ and
$\al\in\left[\frac{1}{3(n+1)},\frac{1}{3}\right]$.
Then,
\begin{equation}
\label{eq:qn_upper_bound}
q(\al)^{n+1}
\le
\frac{1}{4}.
\end{equation}
\end{lem}

\begin{proof}
First, we use the fact that $1/q$ is increasing.
Then, we apply Lemma~\ref{lem:log_q_denom_lower_bound}
with $\al=\frac{1}{3(n+1)}$:
\[
(n+1)\log\frac{1}{q(\al)}
\ge (n+1)\log\frac{1}{q\left(\frac{1}{3(n+1)}\right)}
\ge \frac{6\log(2)\,(n+1)}{3(n+1)}
= 2\log(2).
\]
Multiplying by $-1$ and applying the exponential function we obtain~\eqref{eq:qn_upper_bound}.
\end{proof}

\begin{prop}
\label{prop:f_n_k_contractive}
Let $n\ge 3$ and $j\in\{1,\ldots,m\}$.
Then, $f_{n,j}$ is a contraction on
$\left[\frac{1}{3(n+1)},\,\frac{1}{3}\right]$,
and its Lipschitz coefficient is less than $\frac{11}{18}$.
\end{prop}

\begin{proof}
1. Let $\al\in\left[\frac{1}{3(n+1)},\,\frac{1}{3}\right]$.
First, we have to show that
$f_{n,j}(\al)\in\left[\frac{1}{3(n+1)},\,\frac{1}{3}\right]$.
Propositions~\ref{prop:theta}, \ref{prop:p_q_properties}, and Lemma~\ref{lem:qn_upper_bound}
imply that
\[
0<\tht(\al)<\frac{1}{10\pi},\qquad
0<p(\al)q(\al)^{n+1}\le\frac{1}{8},\qquad
\]
\begin{equation}
\label{eq:arcsin_general_bound}
\frac{1}{\pi}\arcsin\bigl(p(\al)q(\al)^{n+1}\bigr)
\le\frac{\arcsin(1/8)}{\pi}
<\frac{1}{25},
\end{equation}
\begin{equation}
\label{eq:h_general_bound}
|\expterm_{n,j}(\al)|
\le
\frac{1}{25\left(n+\frac{3}{2}\right)}.
\end{equation}
Thus,
\[
f_{n,j}(\al)
=
\frac{j
-\frac{1}{2}
+\frac{r}{3}
-\tht(\al)}%
{n+\frac{3}{2}}
+\expterm_{n,j}(\al)
\le
\frac{m-\frac{1}{2}+\frac{r}{3}+\frac{1}{25}}{n+\frac{3}{2}}
<
\frac{m+\frac{r}{3}}{n}
=\frac{1}{3}.
\]
On the other hand, since
$n+\frac{3}{2}\le\frac{9}{8}(n+1)$,
\[
f_{n,j}(\al)
\ge
\frac{1
-\frac{1}{10\pi}
-\frac{1}{2}
-\frac{1}{25}}%
{n+\frac{3}{2}}
>
\frac{\frac{2}{5}}{\frac{9}{8}(n+1)}
>
\frac{1}{3(n+1)}.
\]
2. Let us estimate the derivative of $f_{n,j}$.
\[
\expterm_{n,j}'(\al)
=
(-1)^{r+j+1}
\frac{\left(p'(\al)+(n+1)\frac{p(\al)q'(\al)}{q(\al)}\right)q(\al)^{n+1}}%
{\left(n+\frac{3}{2}\right)\pi\,\sqrt{1-p(\al)^2 q(\al)^{2(n+1)}}}.
\]
It is easy to verify that the function
$\al\mapsto\frac{p(\al)|q'(\al)|}{q(\al)}$ decreases on $[0,1/3]$
taking values from $\pi\sqrt{3}/2$ to $0$.
Moreover,
$1-p(\al)^2 q(\al)^{2(n+1)}\ge 63/64$.
Therefore,
\[
|\expterm_{n,j}'(\al)|
\le \frac{(n+2)\frac{\sqrt{3}}{2}\cdot\frac{1}{2}}%
{\left(n+\frac{3}{2}\right)\,\frac{\sqrt{63}}{8}}
\le
\frac{10}{9}\cdot\frac{2}{\sqrt{21}}
<\frac{1}{2}.
\]
Finally,
\[
|f_{n,j}'(\al)|
\le
\frac{|\tht'(\al)|}{n+\frac{3}{2}}
+|\expterm_{n,j}'(\al)|
\le \frac{\frac{1}{2}}{3+\frac{3}{2}}
+\frac{1}{2}
<\frac{11}{18}.
\qedhere
\]
\end{proof}

Theorem~\ref{thm:char_equation_as_fixed_point} follows from Proposition~\ref{prop:f_n_k_contractive}.

Figure~\ref{fig:fixed_point} shows the functions $f_{n,j}$ and their fixed points for $n=3$ and $n=10$.

\begin{figure}[htb]
\centering
\includegraphics{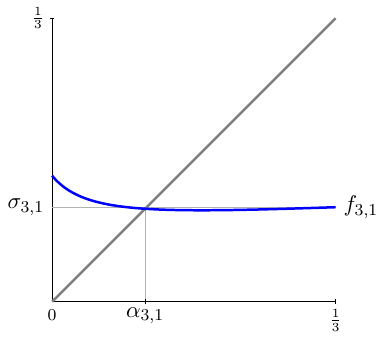}\quad
\includegraphics{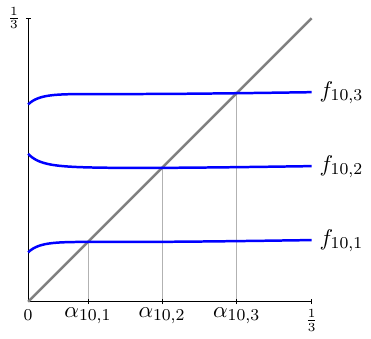}
\caption{Functions $f_{3,1}$ (left) and $f_{10,1}$, $f_{10,2}$, $f_{10,3}$ (right) and their fixed points.
\label{fig:fixed_point}}
\end{figure}

\section{Asymptotic distribution of the eigenvalues}
\label{sec:asymptotic_distribution}

Hirschman~\cite{Hirschman1967} found a general formula for the asymptotic distribution of the eigenvalues of banded Toeplitz matrices,
but that formula is not easy to use in examples.
In this section, basing ourselves on~\eqref{eq:al_nj_minus_si_nj_upper_bound},
we provide a good estimate for the counting function.
As a simple consequence, we get the asymptotic distribution;
that is, Theorem~\ref{thm:asymptotic_distribution}.

Recall that $H_n$ is defined by~\eqref{eq:empirical_distribution_function}.
So, $n H_n$ is the \emph{counting function}
of the strictly positive eigenvalues of $T_n(a)$:
\[
n H_n(x) =
\#\Bigl\{j\in\{1,\ldots,m\}\colon
\quad
0 < \la_{n,j} \le x \Bigr\}.
\]
Zizler, Zuidwijk, Taylor, and Arimoto~\cite[Theorem~2.3]{ZizlerZuidwijkTaylorArimoto2002} proved that for banded Hermitian Toeplitz matrices $T_n(b)$,
the corresponding counting function admits approximations of the form
$|n H_n(x)-n F_b(x)|\le 6d$, 
where $d$ is the half-width of the band
and $F_b$ is the distribution function associated to $b$.
For the non-Hermitian Toeplitz matrices
studied in this paper,
using~\eqref{eq:upper_and_lower_bounds_for_al},
we obtain an approximation of a similar form,
but our function $F$ is not directly related to the values of $a$.
Moreover, in our example,
after appropriate adjustments,
we obtain a smaller upper bound for the error.

\begin{prop}[approximation of the counting function]
\label{prop:estimates_for_the_counting_function}
Let $n\ge3$ and $x\in[0,L]$.
Then
\begin{equation}
\label{eq:estimates_for_the_counting_function}
\left(n+\frac{3}{2}\right)\,F(x)
-\frac{r}{3}-\frac{7}{10}
\le n H_n(x)
\le\left(n+\frac{3}{2}\right)\,F(x)
-\frac{r}{3}+\frac{7}{10}.
\end{equation}
Therefore,
\begin{equation}
\label{eq:approximation_of_the_counting_function}
\left|n H_n(x)
-\left(\left(n+\frac{3}{2}\right)\,F(x)
-\frac{r}{3}\right)
\right|
\le \frac{7}{10}.
\end{equation}
\end{prop}

\begin{proof}
Let $J\eqdef \{j\in\{1,\ldots,m\}\colon\ \la_{n,j}\le x\}$.
Since $F$ is a strictly increasing function,
\[
J = \bigl\{j\in\{1,\ldots,m\}\colon\ \al_{n,j}\le F(x)\bigr\}.
\]
1. We start with proving the upper bound.
If $j\in J$, then,
by~\eqref{eq:upper_and_lower_bounds_for_al},
\[
F(x)
\ge \al_{n,j}
> \frac{j-\frac{7}{10}+\frac{r}{3}}{n+\frac{3}{2}}.
\]
Therefore, the set $J$ can be bounded from above in terms of $F(x)$:
\[
J \subseteq
\bN\cap\left[1, \left(n+\frac{3}{2}\right)F(x) + \frac{7}{10} - \frac{r}{3}\right).
\]
This implies the upper bound in~\eqref{eq:estimates_for_the_counting_function}.

\medskip\noindent
2. To prove the lower bound,
we suppose that
$1\le j\le \left(n+\frac{3}{2}\right)\,F(x) +\frac{3}{10} - \frac{r}{3}$.
Then, by~\eqref{eq:upper_and_lower_bounds_for_al},
\[
\al_{n,j}
<\frac{j-\frac{3}{10}+\frac{r}{3}}{n+\frac{3}{2}}
\le F(x).
\]
Thus, we get a lower bound for the set $J$:
\[
\bN \cap \left[1,\,\left(n+\frac{3}{2}\right)\,F(x)+\frac{3}{10}-\frac{r}{3}\right]
\subseteq J.
\]
It is easy to see that if $A\in\bR$, then $\#(\bN\cap[1,A])>A-1$.
This implies the lower bound
in~\eqref{eq:estimates_for_the_counting_function}.
\end{proof}

\begin{proof}[Proof of Theorem~\ref{thm:asymptotic_distribution}.]
Divide~\eqref{eq:estimates_for_the_counting_function} over $n$:
\[
\left(1+\frac{3}{2n}\right)\,F(x)
-\frac{7}{10n}
-\frac{r}{3n}
\le H_n(x)
\le
\left(1+\frac{3}{2n}\right)\,F(x)
+\frac{7}{10n}
-\frac{r}{3n}.
\]
As $n\to\infty$, the lower and upper estimates tend to $F(x)$,
and we obtain~\eqref{eq:limiting_distribution}.
\end{proof}

\section{A general scheme to solve equations
by asymptotic expansions}
\label{sec:scheme_for_solving_equations}

In this section,
we provide a scheme
for solving equations of the form
$\Omega(z(h),h)=0$
by asymptotic expansions.
We suppose that $h$ is a small parameter,
$\Omega$ has a known asymptotic expansion in three terms with respect to $h$,
and we find a three term expansion of $z(h)$.
We start with a one-term approximation (Proposition~\ref{prop:scheme_solve_equations_by_asymptotic_expansions1}),
then derive expansions with two and three terms (Propositions~\ref{prop:scheme_solve_equations_by_asymptotic_expansions2} and~\ref{prop:scheme_solve_equations_by_asymptotic_expansions3}).
Of course, similar computations can be done for expansions of higher orders.
Our proofs will be based on the mean value theorem 
(in fact, we only use the mean value inequality) and the Taylor theorem.
Such schemes are known among the experts
(see, e.g., de Bruijn~\cite[Chapter~2]{Bruijn1958}),
but it is not easy to find a bibliographic reference corresponding to our needs.
Instead of writing the residue terms as $O(h),O(h^2),\dots$,
we prefer to deal with more explicit upper bounds involving some coefficients $c_1,c_2,\dots$.
The main reason for this hard choice is that the original equation may depend on some parameters, and the coefficients $c_1,c_2,\ldots$ may depend on parameters.

Given a subset $X$ of $\bR$,
we denote by $\closure(X)$
and $\interior(X)$ the closure and interior of $X$,
respectively.
This section is independent on the previous sections, and the notation for the special objects in this section is independent:
for example, the function $\rho_1$
in Proposition~\ref{prop:scheme_solve_equations_by_asymptotic_expansions1}
does not have any relation with the number $\rho$ in Proposition~\ref{prop:roots_via_be}.

\begin{prop}[one-term approximation]
\label{prop:scheme_solve_equations_by_asymptotic_expansions1}
Suppose that $X,D,\Omega,z,\om_0,s_0,\ell,\rho_1,c_1$ are sets, numbers,
and functions 
satisfying the following assumptions.
\begin{enumerate}
\item[A1.]
\mylabel{assumption_X}
$X$ is an interval, $X\subseteq\bR$.
\item[A2.]
\mylabel{assumption_D}
$D$ is a nonempty subset of $\bR$
such that $0\notin D$.
\item[A3.]
\mylabel{assumption_Omega}
$\Omega\colon X\times D\to\bR$.
\item[A4.]
\mylabel{assumption_z}
$z\colon D\to X$
such that for every $h$ in $D$,
$\Omega(z(h),h)=0$.
\item[A5.]
\mylabel{assumption_om0}
$\om_0\colon X\to\bR$
is continuous on $X$
and differentiable on $\interior(X)$.
\item[A6.]
\mylabel{assumption_s0}
$s_0\in\interior(X)$
and $\om_0(s_0)=0$.
\item[A7.]
\mylabel{assumption_om0der1}
$\ell\in (0,+\infty)$
and for every $x$ in $\interior(X)$,
\[
|\om_0'(x)|\ge \ell.
\]
\item[A8.]
\mylabel{assumption_Omega_expansion0}
$\rho_1\colon X\times D\to\bR$
is defined by the following identity which holds for every $x$ in $X$ and $h$ in $D$:
\[
\Omega(x,h)=\om_0(x)+\rho_1(x,h).
\]
\item[A9.]
\mylabel{assumption_rho1_bound}
$c_1\in[0,+\infty)$ such that
\[
\sup_{x\in X}\,\sup_{h\in D}\,
\frac{|\rho_1(x,h)|}{|h|} \le c_1.
\]
\end{enumerate}
Then $z$ can be written in the form
\begin{equation}
\label{eq:z_one_term}
z(h) = s_0 + \eta_1(h),
\end{equation}
where the residue term
$\eta_1\colon D\to\bR$
satisfies
\begin{equation}
\label{eq:z_one_term_error}
\sup_{h\in D} \frac{|\eta_1(h)|}{|h|}
\le \frac{c_1}{\ell}.
\end{equation}
\end{prop}

\begin{proof}
We define $\eta_1\colon D\to\bR$ by
\[
\eta_1(h) \eqdef z(h) - s_0.
\]
For every $h$ in $D$,
using Assumptions~\ref{assumption_X}, \ref{assumption_om0},
\ref{assumption_om0der1}
and applying the mean value theorem,
we get
\[
|\om_0(z(h)) - \om_0(s_0)|
\ge \ell |z(h)-s_0|
= \ell |\eta_1(h)|.
\]
By assumptions~\ref{assumption_z} and~\ref{assumption_s0},
$\om_0(s_0)=0=\Omega(z(h),h)$.
Using these equalities
and Assumptions~\ref{assumption_Omega_expansion0} and~\ref{assumption_rho1_bound},
we can estimate $|\eta_1(h)|/|h|$ from above:
\[
\frac{|\eta_1(h)|}{|h|}
\le 
\frac{|\om_0(z(h)) - \om_0(s_0)|}{\ell|h|}
=
\frac{|\om_0(z(h)) - \Omega(z(h),h)|}{\ell|h|}
=
\frac{|\rho_1(z(h),h)|}{\ell|h|}
\le
\frac{c_1}{\ell}.
\qedhere
\]
\end{proof}

\begin{prop}[two-term approximation]
\label{prop:scheme_solve_equations_by_asymptotic_expansions2}
Suppose that
Assumptions~\ref{assumption_X}--\ref{assumption_om0der1}
from Proposition~\ref{prop:scheme_solve_equations_by_asymptotic_expansions1} are satisfied,
and instead of~\ref{assumption_Omega_expansion0} and~\ref{assumption_rho1_bound},
the following additional assumptions hold.
\begin{enumerate}
\item[A10.]
\mylabel{assumption_D_is_bounded}
$\Delta\in[0,+\infty)$ such that
$D\subseteq[-\Delta,\Delta]$.
\item[A11.]
\mylabel{assumption_om0_om1}
$\om_0,\om_1\colon X\to\bR$
are continuous functions,
$\om_0$ is twice differentiable on
$\interior(X)$,
$\om_1$ is differentiable on
$\interior(X)$.
\item[A12.]
\mylabel{assumption_Omega_expansion1}
$\rho_2\colon X\times D\to\bR$
is defined by the following identity
which holds for every $x$ in $X$ and $h$ in $D$:
\begin{equation}
\label{eq:Omega_two_term}
\Omega(x,h)=\om_0(x)+\om_1(x) h+\rho_2(x,h).
\end{equation}
\item[A13.]
\mylabel{assumption_c02_c10_c11}
$c_{0,2},c_{1,0},c_{1,1}\in[0,+\infty)$ such that
\[
\sup_{x\in\interior(X)}
|\om_0''(x)|
\le c_{0,2},
\quad
\sup_{x\in X}|\om_1(x)|
\le c_{1,0},
\quad
\sup_{x\in\interior(X)}|\om_1'(x)|
\le c_{1,1}.
\]
\item[A14.]
\mylabel{assumption_rho2_bound}
$c_2\in[0,+\infty)$ such that
\[
\sup_{x\in X}\sup_{h\in D}
\frac{|\rho_2(x,h)|}{|h|^2}
\le c_2.
\]
\end{enumerate}
Then $z$ can be written in the form
\begin{equation}
\label{eq:z_two_terms}
z(h) = s_0 + s_1 h + \eta_2(h),
\end{equation}
where the coefficient $s_1$ is defined by
\begin{equation}
\label{eq:s1_def}
s_1 \eqdef -\frac{\om_1(s_0)}{\om_0'(s_0)},
\end{equation}
and the residue term
$\eta_2\colon D\to\bR$
satisfies the following upper bound:
\begin{equation}
\label{eq:eta2_bound}
\sup_{h\in D}
\frac{|\eta_2(h)|}{|h|^2}
\le
\frac{c_2}{\ell}
+\frac{c_{1,1} c_1}{\ell^2}
+\frac{c_{0,2} c_1^2}{2\ell^3},
\qquad
\text{where}
\qquad
c_1 \eqdef c_{1,0} + c_2 \Delta.
\end{equation}
\end{prop}

\begin{proof}
First, we are going to apply Proposition~\ref{prop:scheme_solve_equations_by_asymptotic_expansions1}.
We define $\rho_1\colon X\times D\to\bR$ by
\[
\rho_1(x,h) \eqdef \om_1(x)h+\rho_2(x,h).
\]
By~Assumptions~\ref{assumption_D_is_bounded},
\ref{assumption_c02_c10_c11},
and~\ref{assumption_rho2_bound},
\[
\sup_{x\in X}\sup_{h\in D}
\frac{|\rho_1(x,h)|}{|h|}
\le
\sup_{x\in X}\sup_{h\in D}
\left(|\om_1(x)| + \frac{|\rho_2(x,h)|}{|h|^2}\,|h|\right)
\le
c_{1,0} + c_2 \Delta.
\]
So, Assumptions~\ref{assumption_Omega_expansion0}
and~\ref{assumption_rho1_bound}
from Proposition~\ref{prop:scheme_solve_equations_by_asymptotic_expansions1} are satisfied with
$c_1 \eqdef c_{1,0} + c_2 \Delta$.
We use Proposition~\ref{prop:scheme_solve_equations_by_asymptotic_expansions1}
and represent $z$ in the form
\eqref{eq:z_one_term}.
Moreover, we know that $\eta_1$ satisfies the upper bound~\eqref{eq:z_one_term_error}.

Next, we expand $\Omega$ by \eqref{eq:Omega_two_term}
and rewrite the equation
$\Omega(z(h),h)=0$ as
\begin{equation}
\label{eq:Omega_equation_transformed_1}
\om_0(s_0 + \eta_1(h))
+ \om_1(s_0 + \eta_1(h)) h
+ \rho_2(z(h),h) = 0.
\end{equation}
Our next step is to expand $\om_0$ and $\om_1$ around the point $s_0$.
We set
\begin{align*}
\xi_{0,2}(h)
&\eqdef
\om_0(s_0+\eta_1(h))-\om_0(s_0)-\om_0'(s_0)\eta_1(h),
\\[1ex]
\xi_{1,1}(h)
&\eqdef
\om_1(s_0+\eta_1(h))-\om_1(s_0).
\end{align*}
It follows from these definitions that
\begin{align}
\label{eq:expansion_om02}
\om_0(s_0+\eta_1(h))
&=
\om_0(s_0)
+\om_0'(s_0)\eta_1(h)
+\xi_{0,2}(h),
\\[1ex]
\label{eq:expansion_om11}
\om_1(s_0+\eta_1(h))
&=
\om_1(s_0)+\xi_{1,1}(h).
\end{align}
To estimate the residue terms $\xi_{0,2}$ and $\xi_{1,1}$,
we apply the Taylor theorem to $\om_0$
and the mean value theorem to $\om_1$.
With the help of
Assumptions~\ref{assumption_om0_om1} and~\ref{assumption_c02_c10_c11},
and the estimate from~\eqref{eq:z_one_term_error}, we get
\[
|\xi_{0,2}(h)|
\le \frac{c_{0,2} |\eta_1(h)|^2}{2}
\le \frac{c_{0,2} c_1^2 |h|^2}{2\ell^2},
\qquad
|\xi_{1,1}(h)|
\le c_{1,1} |\eta_1(h)|
\le \frac{c_{1,1} c_1 |h|}{\ell}.
\]
Substituting~\eqref{eq:expansion_om02} and~\eqref{eq:expansion_om11} into~\eqref{eq:Omega_equation_transformed_1}, we get the following equation:
\[
\om_0(s_0)+\om_0'(s_0)\eta_1(h)+\xi_{0,2}(h)
+ \om_1(s_0) h + \xi_{1,1}(h) h
+ \rho_2(z(h),h) = 0.
\]
Due to Assumption~\ref{assumption_s0},
we can omit the zero term $\om_0(s_0)$.
Next, we solve the equation for
$\eta_1(h)$:
\[
\eta_1(h)
=
-\frac{\om_1(s_0)}{\om_0'(s_0)} h
+ \eta_2(h),
\]
where
\[
\eta_2(h)
\eqdef
-
\frac{\rho_2(x,h)+\xi_{0,2}(h)+\xi_{1,1}(h)h}{\om_0'(s_0)}.
\]
Using the above upper bounds
for $\rho_2$, $\xi_1$, and $\xi_2$,
we finally estimate $\eta_2$:
\[
\sup_{h\in D}\frac{|\eta_2(h)|}{|h|^2}
\le
\frac{c_2 + \frac{c_{1,1} c_1}{\ell}
+ \frac{c_{0,2} c_1^2}{2\ell^2}}{\ell}
=
\frac{c_2}{\ell}
+\frac{c_{1,1} c_1}{\ell^2}
+\frac{c_{0,2} c_1^2}{2\ell^3}.
\qedhere
\]
\end{proof}

\begin{prop}[three-term approximation]
\label{prop:scheme_solve_equations_by_asymptotic_expansions3}
Suppose that Assumptions \ref{assumption_X}--\ref{assumption_om0der1}
from Proposition~\ref{prop:scheme_solve_equations_by_asymptotic_expansions1}
are satisfied,
and the following additional assumptions hold.
\begin{enumerate}
\item[A10.]
$\Delta\in[0,+\infty)$ such that $D\subseteq[-\Delta,\Delta]$.
\item[A15.]
\mylabel{assumption_om0_om1_om2}
$\om_0,\om_1,\om_2\colon X\to\bR$
are continuous functions,
$\om_0$ is three times differentiable on $\interior(X)$,
$\om_1$ is twice differentiable on $\interior(X)$,
and $\om_2$ is differentiable on
$\interior(X)$.
\item[A16.]
$\rho_3\colon X\times D\to\bR$
is defined by the following identity which holds for every $x$ in $X$
and $h$ in $D$:
\begin{equation}
\mylabel{eq:Omega_three_terms}
\Omega(x,h)
=
\om_0(x)
+\om_1(x) h
+\om_2(x) h^2
+\rho_3(x,h).
\end{equation}
\item[A17.]
\mylabel{assumption_c03_c12_c20_c21}
Assumption~\ref{assumption_c02_c10_c11}
from Proposition~\ref{prop:scheme_solve_equations_by_asymptotic_expansions2} is satisfied
and
$c_{0,3},c_{1,2},c_{2,0},c_{2,1}\in[0,+\infty)$ such that
\[
\sup_{x\in\interior(X)}
|\om_0'''(x)|
\le c_{0,3},
\quad
\sup_{x\in \interior(X)}|\om_1''(x)|
\le c_{1,2},
\]
\[
\sup_{x\in X}|\om_2(x)|
\le c_{2,0},
\quad
\sup_{x\in\interior(X)}|\om_2'(x)|
\le c_{2,1}.
\]
\item[A18.]
\mylabel{assumption_rho3_bound}
$c_3\in[0,+\infty)$ such that
\[
\sup_{x\in X}\sup_{h\in D}
\frac{|\rho_3(x,h)|}{|h|^3}
\le c_3.
\]
\end{enumerate}
Then $z$ can be written in the form
\begin{equation}
\label{eq:z_three_terms}
z(h)
= s_0 + s_1 h + s_2 h^2
+ \eta_3(h),
\end{equation}
where
\begin{equation}
\label{eq:s1_s2_def}
s_1 \eqdef -\frac{\om_1(s_0)}{\om_0'(s_0)},
\qquad
s_2 \eqdef -\frac{1}{\om_0'(s_0)}
\left(\frac{\om_0''(s_0)}{2}s_1^2
+\om_1'(s_0)s_1
+\om_2(s_0)
\right),
\end{equation}
and
\begin{equation}
\label{eq:eta3_bound}
\sup_{h\in D}
\frac{|\eta_3(h)|}{|h|^3}
\le
\frac{c_{0,2} c_{1,0} k}{\ell^2}
+\frac{c_{0,3}c_1^3}{6\ell^4}
+\frac{c_{1,1}k}{\ell}
+\frac{c_{1,2}c_1^2}{2\ell^3}
+\frac{c_{2,1}c_1}{\ell^2}
+\frac{c_3}{\ell},
\end{equation}
where
\begin{equation}
\label{eq:additional_constants_definition}
c_2\eqdef c_{2,0} + c_3 \Delta,\qquad
c_1 \eqdef c_{1,0}+c_2\Delta,\qquad
k \eqdef \frac{c_2}{\ell}
+\frac{c_{1,1} c_1}{\ell^2}
+\frac{c_{0,2} c_1^2}{2\ell^3}.
\end{equation}
\end{prop}

\begin{proof}
First, we are going to apply Proposition~\ref{prop:scheme_solve_equations_by_asymptotic_expansions2}
with
\[
\rho_2(x,h)
\eqdef \om_2(x)h^2+\rho_3(x,h).
\]
Using Assumptions~\ref{assumption_D_is_bounded}
and \ref{assumption_om0_om1_om2}--\ref{assumption_rho3_bound},
we get the following estimate for $\rho_2$:
\[
\sup_{x\in X}\sup_{h\in D}
\frac{|\rho_2(x,h)|}{|h|^2}
\le
\sup_{x\in X}\sup_{h\in D}
\left(|\om_2(x)| + \frac{|\rho_3(x,h)|}{|h|^3}\,|h|\right)
\le c_2,
\quad
\text{where}
\quad
c_2 \eqdef c_{2,0} + c_3 \Delta.
\]
All assumptions from Proposition~\ref{prop:scheme_solve_equations_by_asymptotic_expansions2} are satisfied, 
and we represent $z$ in the form
\eqref{eq:z_two_terms}.
Moreover, by~\eqref{eq:eta2_bound}
we know that
\begin{equation}
\label{eq:eta2_short_bound}
\sup_{h\in D}
\frac{|\eta_2(h)|}{|h|^2} \le k.
\end{equation}
Using the smoothness of
the functions
$\om_0$, $\om_1$, and $\om_2$,
we expand them around the point $s_0$:
\begin{align*}
\om_2(z(h))
&=
\om_2(s_0 + s_1 h + \eta_2(h))
=\om_2(s_0)
+\xi_{2,1}(h),
\\[1ex]
\om_1(z(h))
&=
\om_1(s_0 + s_1 h + \eta_2(h))
=\om_1(s_0)
+\om_1'(s_0) (s_1 h + \eta_2(h))
+\xi_{1,2}(h),
\\[1ex]
\om_0(z(h))
&=
\om_0(s_0)
+\om_0'(s_0) (s_1 h + \eta_2(h))
+\frac{\om_0''(s_0)}{2}(s_1 h+\eta_2(h))^2
+\xi_{0,3}(h).
\end{align*}
The residue terms $\xi_{2,1}$, $\xi_{1,2}$, and $\xi_{0,3}$ are defined by these identities.
Notice that $s_1 h + \eta_2(h)$ is just $\eta_1(h)$ from the proof of Proposition~\ref{prop:scheme_solve_equations_by_asymptotic_expansions2},
and it is absolutely bounded by $c_1/\ell$.
Now, we estimate $\xi_{2,1}$, $\xi_{1,2}$, and $\xi_{0,3}$ by using Taylor's theorem and Assumption~\ref{assumption_c03_c12_c20_c21}:
\begin{equation}
\label{eq:xi21_xi12_x03_bound}
\frac{|\xi_{2,1}(h)|}{|h|}
\le \frac{c_{2,1} c_1}{\ell},
\qquad
\frac{|\xi_{1,2}(h)|}{|h|^2}
\le \frac{c_{1,2} c_1^2}{2\ell^2},
\qquad
\frac{|\xi_{0,3}(h)|}{|h|^3}
\le \frac{c_{0,3}c_1^3}{6\ell^3}.
\end{equation}
Next, we expand $\Omega$ by \eqref{eq:Omega_three_terms},
substitute the expansions
of $\om_0(z(h))$, $\om_1(z(h))$, $\om_2(z(h))$ written above,
and transform the equation
$\Omega(z(h),h)=0$ to the following form:
\begin{align*}
&\om_0(s_0)
+\om_0'(s_0) (s_1 h + \eta_2(h))
+\frac{\om_0''(s_0)}{2}(s_1 h+\eta_2(h))^2
+\xi_{0,3}(h)
\\
&
\quad
+\Bigl(\om_1(s_0)
+\om_1'(s_0) (s_1 h + \eta_2(h))
+\xi_{1,2}(h)\Bigr)h
+
\Bigl(\om_2(s_0)+\xi_{2,1}(h)\Bigr)h^2
+\rho_3(x,h)
=0.
\end{align*}
In this equation,
we regroup the summands 
in the following way:
\begin{align*}
&\Bigl(\om_0(s_0)
+(\om_0'(s_0)s_1
+ \om_1(s_0))h
\Bigr)
+\om_0'(s_0)\eta_2(h)
+\left(
\frac{\om_0''(s_0)}{2} s_1^2 
+\om_1'(s_0)s_1
+\om_2(s_0)
\right)h^2
\\[1ex]
&\qquad
+\om_0''(s_0)s_1 \eta_2(h) h
+\frac{\om_0''(s_0)}{2}\,\eta_2(h)^2
+\xi_{0,3}(h)
\\[1ex]
&\qquad
+\om_1'(s_0)\eta_2(h)h
+\xi_{1,2}(h)h
+\xi_{2,1}(h)h^2
+\rho_3(x,h)
=0.
\end{align*}
The first group of terms is zero,
by Assumption~\ref{assumption_s0}
and the definition of $s_1$.
Now, we keep the term $\om_0'(s_0)\eta_2(h)$
in the left-hand side of the equation,
pass the other terms to the right-hand side and divide by $\om_0'(s_0)$.
Taking into account the definition \eqref{eq:s1_s2_def} of $s_2$, we get
\[
\eta_2(h)=s_2 h^2 + \eta_3(h),
\]
where we define $\eta_3\colon D\to\bR$ as
\begin{align*}
\eta_3(h)
&
\eqdef
-\frac{\frac{\om_0''(s_0)}{2} \eta_2(h)
\left(s_1 h + \eta_1(h)\right)
+\xi_{0,3}(h)}{\om_0'(s_0)}
\\[1ex]
&\qquad\qquad
-\frac{\om_1'(s_0)\eta_2(h)h
+\xi_{1,2}(h)h
+\xi_{2,1}(h)h^2
+\rho_3(x,h)}{\om_0'(s_0)}.
\end{align*}
Thereby, we have written $z(h)$ in the form~\eqref{eq:z_three_terms},
and we are left to get an upper bound for $|\eta_3|$.
We will apply Assumptions~\ref{assumption_om0_om1_om2}
and~\ref{assumption_rho3_bound},
the upper bounds~\eqref{eq:xi21_xi12_x03_bound}
and~\eqref{eq:eta2_short_bound},
and the following estimate for $s_1$:
\[
|s_1|
= \frac{|\om_1(s_0)|}{|\om_0'(s_0)|}
\le \frac{c_{1,0}}{\ell}
\le \frac{c_1}{\ell}.
\]
Then, we obtain
\begin{align*}
\frac{|\eta_3(h)|}{|h|^3}
&\le
\frac{\frac{c_{0,2} c_{1,0} k}{\ell}
+\frac{c_{0,3}c_1^3}{6\ell^3}
+c_{1,1}k
+\frac{c_{1,2}c_1^2}{2\ell^2}
+\frac{c_{2,1}c_1}{\ell}
+c_3}{\ell},
\end{align*}
which is equivalent to~\eqref{eq:eta3_bound}.
\end{proof}

The following particular case
is similar to the asymptotic expansions studied in~\cite{BogoyaBottcherGrudskyMaximenko2015,BarreraBoettcherGrudskyMaximenko2018,BoettcherGrudskyMaksimenko2010},
where $h$ was of the form
$h=1/n$, $h=1/(n+1)$, or $h=1/(n+2)$.
We will apply this particular case in Section~\ref{sec:asympt_expansion_not_too_close_to_the_origin}.

\begin{cor}
\label{cor:scheme_particular_case3}
Under the assumptions of Proposition~\ref{prop:scheme_solve_equations_by_asymptotic_expansions3},
if $\om_0$ is a first degree polynomial of the form $\om_0(x)=x-s_0$,
and $\om_2,\rho_3$ are the zero functions,
then
\[
z(h)
= s_0
-\om_1(s_0)h
+\om_1(s_0)\om_1'(s_0)h^2
+\eta_3(h),
\]
where the residue term admits the following upper bound:
\[
\sup_{h\in D}\frac{|\eta_3(h)|}{|h|^3}
\le
c_{1,0} c_{1,1}^2
+\frac{c_{1,0}^2 c_{1,2}}{2}.
\]
\end{cor}

\begin{proof}
In the notation of Proposition~\ref{prop:scheme_solve_equations_by_asymptotic_expansions3},
we have $\ell=1$,
$c_2=0$, $c_1=c_{1,0}$, $k=c_{1,0}c_{1,1}$.
\end{proof}

\section{Asymptotic expansion not too close to the origin}
\label{sec:asympt_expansion_not_too_close_to_the_origin}

In this section,
we prove Theorem~\ref{thm:lambda_asympt_not_too_close_to_the_origin}.
First, we give an elementary general scheme to estimate the distance between the fixed points of two functions, supposing that at least one of them is contractive.

\begin{prop}
\label{prop:two_contractive_functions}
Let $(X,d)$ be a complete metric space,
$f_1\colon X\to X$
and $t\in X$ such that $f_1(t)=t$.
Furthermore,
let
$f_2\colon X\to X$
be a contractive function,
$\ka$ be a Lipschitz coefficient of $f_2$
satisfying $0\le\ka<1$,
and $u\in X$ such that $f_2(u)=u$.
Then,
\[
d(t,u) \le \frac{d(f_1(t),f_2(t))}{1-\ka}.
\]
\end{prop}

\begin{proof}
Apply the triangle inequality
and the Lipschitz property of $f_2$:
\[
d(t,u)=d(f_1(t),f_2(u))
\le d(f_1(t),f_2(t))+d(f_2(t),f_2(u))
\le d(f_1(t),f_2(t))+\ka\,d(t,u).
\]
Passing the term $\ka\,d(t,u)$ to the left-hand side and dividing by $1-\ka$ we obtain the result.
\end{proof}

Beside the exact equation $\al=f_{n,j}(\al)$, where $f_{n,j}$ is defined by
\[
f_{n,j}(\al)
=
\si_{n,j}
-\frac{\tht(\al)}{n+\frac{3}{2}}
+\expterm_{n,j}(\al),
\]
we consider the ``approximate equation''
$\al=\widetilde{f}_{n,j}(\al)$, where
\begin{equation}
\label{eq:fapprox_def}
\widetilde{f}_{n,j}(\al)
\eqdef
\si_{n,j}
-\frac{\tht(\al)}{n+\frac{3}{2}}.
\end{equation}
We are going to show that $\widetilde{f}_{n,j}$
has a unique fixed point
$\widetilde{\al}_{n,j}$ on $[0,1/3]$.
Moreover,
if $j\ge\log(n)$,
then $\widetilde{\al}_{n,j}$
is close to $\al_{n,j}$.

\begin{prop}
\label{prop:gnj_is_contraction}
Let $n\ge 3$ and $j\in\{1,\ldots,m\}$.
Then, $\widetilde{f}_{n,j}$ is a contraction
on $[0,1/3]$.
Moreover, $\widetilde{f}_{n,j}$ is a contraction
on $\left[\frac{1}{3\left(n+\frac{3}{2}\right)},\frac{1}{3}\right]$.
\end{prop}

\begin{proof}
First, we verify that $\widetilde{f}_{n,j}(\al)$ belongs to $[0,1/3]$ for every $\al$ in $[0,1/3]$:
\[
\frac{1}{3\left(n+\frac{3}{2}\right)}
< \frac{1-\frac{1}{10\pi}-\frac{1}{2}}{n+\frac{3}{2}}
\le \widetilde{f}_{n,j}(\al)
< \frac{m+\frac{r}{3}}{n+\frac{3}{2}}
< \frac{1}{3}.
\]
Furthermore,
for all $\al$ in $[0,1/3]$
we have $|\tht'(\al)|<1/2$ and
\[
|\widetilde{f}_{n,j}'(\al)|
\le\frac{1}{2\left(n+\frac{3}{2}\right)}
=\frac{1}{2n+3}.
\]
Thus, $\widetilde{f}_{n,j}$ is Lipschitz continuous with coefficient $\le 1/(2n+3)$.
\end{proof}

For every $j$ in $\{1,\ldots,m\}$,
we denote the fixed point of
$\widetilde{f}_{n,j}$
by $\widetilde{\al}_{n,j}$.

\begin{lem}
\label{lem:q_power_n_upper_bound_far_from_zero}
If $n\ge 32$
and $\al\in[0,1/3]$ such that
\[
\al
\ge 
\frac{\log n}{2n},
\]
then
\[
q(\al)^{n+1}
< \frac{1}{\left(n+\frac{3}{2}\right)^2}.
\]
\end{lem}

\begin{proof}
Since $n\ge 32$,
\[
\frac{\log n}{\log\left(n+\frac{3}{2}\right)}
\ge\frac{\log 32}{\log\frac{67}{2}}
>\frac{63}{64},\qquad
\frac{\log n}{2n}
<\frac{1}{12}.
\]
By Lemma~\ref{lem:log_q_denom_lower_bound},
\[
-\log(q(\al))
\ge -\log\left(q\left(\frac{\log n}{2n}\right)\right)
\ge \frac{3\log(2)\,\log(n)}{n}.
\]
Therefore,
\[
-(n+1)\log(q(\al))
\ge
\frac{63\cdot 3\log 2}{64}
\cdot
\log\left(n+\frac{3}{2}\right)
> 2\log\left(n+\frac{3}{2}\right),
\]
and the result follows.
\end{proof}

\begin{lem}
\label{lem:Rnj_upper_bound_far_from_zero}
If $n\ge 32$
and $j\in\{1,\ldots,m\}$
such that
$j\ge \log n$,
then
\begin{equation}
\label{eq:Rnj_upper_bound_far_from_zero}
|\expterm_{n,j}(\al_{n,j})|
<
\frac{1}{5\left(n+\frac{3}{2}\right)^3}.
\end{equation}
\end{lem}

\begin{proof}
By~\eqref{eq:al_nj_minus_si_nj_upper_bound},
\[
\al_{n,j}
>
\frac{j-\frac{1}{2}+\frac{r}{3}-\frac{1}{5}}{n+\frac{3}{2}}
\ge
\frac{\log n-\frac{7}{10}}{\frac{67}{64}n}
=
\frac{\left(1-\frac{7}{10\log n}\right)\log n}{\frac{67}{64}n}
>
\frac{\log n}{2n}.
\]
By Lemma~\ref{lem:q_power_n_upper_bound_far_from_zero},
\[
q(\al_{n,j})^{n+1}
<\frac{1}{\left(n+\frac{3}{2}\right)^2}.
\]
Furthermore, $\arcsin x\le\frac{2}{\sqrt{3}} x$ for $0\le x\le 1/2$, and
\[
\expterm_{n,j}(\al_{n,j})
=
\frac{1}{\left(n+\frac{3}{2}\right)\pi}\arcsin\left(p(\al_{n,j})q(\al_{n,j})^{n+1}\right)
<
\frac{1}{\pi\sqrt{3}}
\,\frac{1}{\left(n+\frac{3}{2}\right)^3}
<
\frac{1}{5\left(n+\frac{3}{2}\right)^3}.
\qedhere
\]
\end{proof}

\begin{prop}
\label{prop:al_minus_widetilde_al}
Let $n\ge 16$ and $j\in\{1,\ldots,m\}$ such that $j\ge \log n$.
Then,
\[
|\al_{n,j}-\widetilde{\al}_{n,j}|
<\frac{1}{4\left(n+\frac{3}{2}\right)^3}.
\]
\end{prop}

\begin{proof}
By Theorem~\ref{thm:char_equation_as_fixed_point} and Proposition~\ref{prop:gnj_is_contraction},
$f_{n,j}$ and $\widetilde{f}_{n,j}$ are contractions
on $\left[\frac{1}{3\left(n+\frac{3}{2}\right)},\frac{1}{3}\right]$.
The Lipschitz coefficient of
$\widetilde{f}_{n,j}$ is bounded by
\[
\frac{1}{2n+3}
\le \frac{1}{131}.
\]
By Proposition~\ref{prop:two_contractive_functions},
\[
|\al_{n,j}-\widetilde{\al}_{n,j}|
\le\frac{1}{1-\frac{1}{131}}\,
|f_{n,j}(\al_{n,j})
-\widetilde{f}_{n,j}(\al_{n,j})|
\le
\frac{5}{4}
|\expterm_{n,j}(\al_{n,j})|
<
\frac{1}{4\left(n+\frac{3}{2}\right)^3}.
\qedhere
\]
\end{proof}

\begin{prop}
\label{prop:alphas_asympt_expansion_not_too_close_to_the_origin}
As $n\to\infty$ and $j\ge\log n$,
the numbers $\al_{n,j}$ have the following asymptotic expansion:
\begin{equation}
\label{eq:alphas_asympt_expansion_not_too_close_to_the_origin}
\al_{n,j}
=\si_{n,j}
-\frac{\tht(\si_{n,j})}{n+\frac{3}{2}}
+\frac{\tht(\si_{n,j})\tht'(\si_{n,j})}{\left(n+\frac{3}{2}\right)^2}
+O\left(\frac{1}{\left(n+\frac{3}{2}\right)^3}\right).
\end{equation}
\end{prop}

\begin{proof}
Due to Proposition~\ref{prop:al_minus_widetilde_al}, it is sufficient to obtain an asymptotic expansion of
$\widetilde{\al}_{n,j}$.
Since $\tht$ is a smooth function not depending on $n$,
formula~\eqref{eq:alphas_asympt_expansion_not_too_close_to_the_origin}
can be easily proved by methods from~\cite{BogoyaBottcherGrudskyMaximenko2015,BarreraBoettcherGrudskyMaximenko2018,BoettcherGrudskyMaksimenko2010}.

Another way to obtain the same result is to use the scheme from Section~\ref{sec:scheme_for_solving_equations}.
Introducing the parameters
\[
d = \si_{n,j},\qquad
h=\frac{1}{n+\frac{3}{2}},
\]
we rewrite the equation
$\al=\widetilde{f}_{n,j}(\al)$
defining
$\widetilde{\al}_{n,j}$
in the form
$\Omega_d(\al,h)=0$, where
\[
\Omega_d(\al,h)
\eqdef
\al - d + \tht(\al) h.
\]
Finally, we apply 
Corollary~\ref{cor:scheme_particular_case3}
with the following functions
instead of $\om_0$ and $\om_1$:
\[
\om_{d,0}(\al)\eqdef \al - d,
\qquad
\om_{d,1}(\al)\eqdef \tht(\al).
\]
Notice that $s_{d,0}\eqdef d$
is the unique zero
of the function $\om_{d,0}$.
It is important that the functions
$\om_{d,0}$, $\om_{d,1}$
and their derivatives
admit uniform bounds 
(not depending on $d$)
when $d$ belongs to $[0,1/3]$.
Therefore, the residue term in~\eqref{eq:alphas_asympt_expansion_not_too_close_to_the_origin}
admit a bound not depending on $d$. 
\end{proof}

\begin{proof}[Proof of Theorem~\ref{thm:lambda_asympt_not_too_close_to_the_origin}]
Follows from~\eqref{eq:alphas_asympt_expansion_not_too_close_to_the_origin}
and $\la_{n,j}=G(\al_{n,j})$
by using the Taylor expansion of $G$ at the point $\si_{n,j}$.
See similar computations in \cite{BogoyaBottcherGrudskyMaximenko2015,BarreraBoettcherGrudskyMaximenko2018,BoettcherGrudskyMaksimenko2010}.
\end{proof}

The first derivatives of $G$ participating in Theorem~\ref{thm:lambda_asympt_not_too_close_to_the_origin} can be computed by the following explicit formulas,
with $t\eqdef \tan\left(\frac{\pi}{3}-\pi\al\right)$:
\begin{equation}
\label{eq:G_derivatives}
G(\al)
=L\,
\frac{1-\frac{1}{3}t^2}%
{(1+t^2)^{2/3}},
\quad
G'(\al)
=2L\pi\,
\frac{t+\frac{1}{9}t^3}%
{(1+t^2)^{2/3}},
\quad
G''(\al)
=-2L\pi^2\,
\frac{1+\frac{5}{27}t^4}%
{(1+t^2)^{2/3}}.
\end{equation}

\begin{rem}
With $h=\frac{1}{n+\frac{3}{2}}$,
the equation $\al=\widetilde{f}_{n,j}(\al)$
for $\widetilde{\al}_{n,j}$
can also be written in the form
\[
\al
+
\left(-j + \frac{1}{2} - \frac{r}{3}
+ \tht(\al)\right) h = 0.
\]
We can apply Corollary~\ref{cor:scheme_particular_case3}
with
\[
\widetilde{\om}_{r,j,0}(\al)
\eqdef\al,
\qquad
\widetilde{\om}_{r,j,1}(\al)
\eqdef
-j+\frac{1}{2}-\frac{r}{3}+\tht(\al),
\qquad
\widetilde{s}_{r,j,0}\eqdef0.
\]
By Proposition~\ref{prop:theta},
we obtain
\[
\widetilde{\om}_{r,j,1}(0)
=-j+\frac{1}{2}-\frac{r}{3},
\qquad
\widetilde{\om}_{r,j,1}'(0)
=\tht'(0)=\frac{1}{2},
\qquad
\sup_{0\le \al\le \frac{1}{3}}
|\widetilde{\om}_{r,j,1}(\al)|
\le c_{r,j,1,0} \eqdef j+1.
\]
Thereby, we get the asymptotic expansion
\begin{equation}
\label{eq:altilde_near_zero}
\widetilde{\al}_{n,j}
=
\frac{j-\frac{1}{2}+\frac{r}{3}}{n+\frac{3}{2}}
-
\frac{j-\frac{1}{2}+\frac{r}{3}}{2\left(n+\frac{3}{2}\right)^2}
+
O\left(\frac{j^2}{\left(n+\frac{3}{2}\right)^3}\right).
\end{equation}
Notice that
\[
\frac{1}{n+2}
=\frac{1}{n+\frac{3}{2}+\frac{1}{2}}
=\frac{1}{\left(n+\frac{3}{2}\right)
\left(1+\frac{1}{2\left(n+\frac{3}{2}\right)}\right)}
=\frac{1}{n+\frac{3}{2}}
-\frac{1}{2\left(n+\frac{3}{2}\right)^2}
+O\left(\frac{1}{\left(n+\frac{3}{2}\right)^3}\right).
\]
Therefore,
the expansion~\eqref{eq:altilde_near_zero}
is equivalent to the following one:
\begin{equation}
\label{eq:altilde_near_zero1}
\widetilde{\al}_{n,j}
=
\frac{j-\frac{1}{2}+\frac{r}{3}}{n+2}
+
O\left(\frac{j^2}{(n+2)^3}\right).
\end{equation}
Since the residue terms in~\eqref{eq:altilde_near_zero}
and
\eqref{eq:altilde_near_zero1}
depend on $j$,
the asymptotic expansions
\eqref{eq:altilde_near_zero}
and~\eqref{eq:altilde_near_zero1}
are much less precise than~\eqref{eq:alphas_asympt_expansion_not_too_close_to_the_origin}
for $j\ge \log n$.
\end{rem}

\begin{rem}
\label{rem:decreasing_ordering}
The initial approximation
$\si_{n,j}\eqdef\frac{j-\frac{1}{2}+\frac{r}{3}}{n+\frac{3}{2}}$
depends on $r=\remainder(n,3)$.
Numbering the eigenvalues in the inverse order,
we can avoid this dependence:
\[
\si_{n,m+1-k}
=\frac{m+1-k-\frac{1}{2}+\frac{r}{3}}{n+\frac{3}{2}}
=\frac{\frac{n}{3}+\frac{1}{2}}{n+\frac{3}{2}}
-\frac{k}{n+\frac{3}{2}}
=\frac{1}{3}-\frac{k}{n+\frac{3}{2}}.
\]
\end{rem}

So, the asymptotic formula from Theorem~\ref{thm:lambda_asympt_not_too_close_to_the_origin} can be rewritten in the following form, which is more convenient for the eigenvalues close to $L$.

\begin{cor}
\label{cor:lambda_asympt_not_too_close_to_the_origin_reverse}
For every $n$ large enough
and every $k$ in $\{1,\ldots,m\}$, if $k\le m+1-\frac{3}{2}\log(n)$, then
\begin{equation}
\label{eq:lambda_asympt_not_too_close_to_the_origin_reverse}
\la_{n,m+1-k}
=\overleftarrow{A}_0\left(\frac{k}{n+\frac{3}{2}}\right)
+\frac{\overleftarrow{A}_1\left(\frac{k}{n+\frac{3}{2}}\right)}{n+\frac{3}{2}}
+\frac{\overleftarrow{A}_2\left(\frac{k}{n+\frac{3}{2}}\right)}{\left(n+\frac{3}{2}\right)^2}
+O\left(\frac{1}{\left(n+\frac{3}{2}\right)^3}\right),
\end{equation}
where
$\overleftarrow{A}_\ell(x)\eqdef A_\ell\left(\frac{1}{3}-x\right)$, $\ell\in\{0,1,2\}$.
\end{cor}

\section{Asymptotic expansion close to the origin}
\label{sec:close_to_the_origin}

In this section,
we derive an asymptotic expansion of the eigenvalues $\la_{n,j}$ as $n\to\infty$ and $j\le\log n$.
Now, the third term from Widom's formula~\eqref{eq:Widom_det}
is comparable with the first two terms,
and we have to take it into account.
In other words, we cannot drop the term $\expterm_{n,j}$
in the main equation~\eqref{eq:maineq}.

As we show below,
it is natural to work with the small parameter $h=1/(n+2)$
and treat $\ga=(n+2)\al$ as a new unknown variable in the characteristic equation.
Working with $(n+2)\al$ instead of
$n\al$ or $\left(n+\frac{3}{2}\right)\al$
allows us to reduce the number of terms in some equations,
as it can be seen from~\eqref{eq:altilde_near_zero1}
or from the computations below.

In this section, we denote by $\ga_{n,j}$
the product $(n+2)\al_{n,j}$.

We recall that $\al_{n,j}$ is the solution of the main equation~\eqref{eq:maineq}.
Equivalently, $\al_{n,j}$ satisfies
\begin{equation}
\label{eq:maineq1}
\left(n+\frac{3}{2}\right)\al_{n,j}
+\tht(\al_{n,j})
=j-\frac{1}{2}+\frac{r}{3}
+\frac{(-1)^{r+j+1}}{\pi}
\arcsin\left(p(\al_{n,j})q(\al_{n,j})^{n+1}\right).
\end{equation}

We start with an approximation for
$\tht(\al)$.
Recall that $\tht$ is defined by~\eqref{eq:def_theta}.

\begin{lem}
\label{lem:approximation_theta_near_origin}
The following asymptotic expansion hold for $\al$ in $[0,1/3]$:
\begin{equation}
\label{eq:tht_approximation_near_0}
\tht(\al) = 
\frac{1}{2}\al-\frac{2\sqrt{3}}{3}\pi\al^2 + O(\al^3).
\end{equation}
\end{lem}

\begin{proof}
We use the Taylor expansion of $\arctan$
around a point $t_0$:
\begin{equation}
\label{eq:actan_expansion}
\arctan(t)
=\arctan(t_0)
+\frac{1}{t_0^2+1}(t-t_0)
-\frac{t_0}{(t_0^2+1)^2}(t-t_0)^2
+O((t-t_0)^3).
\end{equation}
In our case,
$t_0=\frac{\sqrt{3}}{3}$
and
$t=\frac{\tan\left(\frac{\pi}{3}-\pi\al\right)}{3}$.
Using the Taylor expansion of $\tan$ around the point $\pi/3$, we get
\begin{equation}
\label{eq:our_tan_expansion}
t-t_0
=
\frac{1}{3}
\left(
\tan\left(\frac{\pi}{3}-\pi\al\right)
-\sqrt{3}
\right)
=
\frac{\sqrt{3}}{3}
-\frac{4\pi}{3}\,\al
+\frac{4\sqrt{3}\,\pi^2}{3}\,\al^2
+O(\al^3).
\end{equation}
Substituting~\eqref{eq:our_tan_expansion}
into~\eqref{eq:actan_expansion},
we obtain~\eqref{eq:tht_approximation_near_0}.
\end{proof}

Due to Lemma~\ref{lem:approximation_theta_near_origin},
the left-hand side of~\eqref{eq:maineq1}
transforms to $(n+2)\al_{n,j}$ plus some residue term.
Motivated by this observation,
we use the change of variables
\[
\ga = (n+2)\al.
\]
Also, instead of the large parameter,
we prefer working with the small parameter
\[
h = \frac{1}{n+2}.
\]
We define
\begin{equation}
p_1(\al) \eqdef \frac{p(\al)}{q(\al)},
\qquad
\text{that is,}
\qquad
p_1(\al) =
\frac{2\sin\left(\frac{\pi}{3}-\pi\al\right)}{\sqrt{9+\tan^2\left(\frac{\pi}{3}-\pi\al\right)}}.
\end{equation}
Next, for every $r$ in $\{0,1,2\}$ and every $j$ in $\bN$, we define
$\Omega_{r,j}\colon
\left[0,\frac{1}{3h}-\frac{2}{3}\right] \times (0,1]\to\bR$ by
\begin{equation}
\label{eq:def_Omega}
\Omega_{r,j}(\ga,h)
\eqdef
\ga
+\tht(\ga h)-\frac{\ga h}{2}
- j+\frac{1}{2}-\frac{r}{3}
+\frac{(-1)^{r+j}}{\pi}
\arcsin\left(
p_1(\ga h)q(\ga h)^{1/h}\right).
\end{equation}
The technical conditions
$0\le \ga \le \frac{1}{3h}-\frac{2}{3}$
are equivalent to $0\le \al \le 1/3$.
They guarantee that the arguments of $\tht$, $p_1$, and $q$ in the right-hand side of~\eqref{eq:def_Omega}
belong to $[0,1/3]$.

For every $n$ in $\bN$ and every $j$ in $\{1,\ldots,m\}$, we define
\[
\ga_{n,j}
\eqdef (n+2)\al_{n,j}.
\]

\begin{prop}[main equation near the origin in terms of the variables $\ga$ and $h$]
\label{prop:main_equation_near_the_origin_in_terms_of_ga_and_h}
Let $n\in\bN$,
$j\in\{1,\ldots,m\}$,
and $h\eqdef 1/(n+2)$.
Then, $\ga_{n,j}$ is the unique number
$\ga$ in $[0,\frac{1}{3h}-\frac{2}{3}]$ satisfying 
\begin{equation}
\label{eq:equation_near_zero_in_terms_of_Omega}
\Omega_{n,j}(\ga,h)=0.
\end{equation}
Moreover, the following estimates hold:
\begin{equation}
\label{eq:rough_bounds_for_ga}
j - \frac{2}{3}
\le \ga_{n,j}
\le j + \frac{1}{2}.
\end{equation}
\end{prop}

\begin{proof}
After the changes of variables $h=1/(n+2)$ and $\ga=(n+2)\al$,
the equation~\eqref{eq:maineq1}
takes the form
\begin{equation}
\label{eq:main_equation_for_ga}
\ga
=
\frac{\ga h}{2} - \tht(\ga h)
+j-\frac{1}{2}+\frac{r}{3}
+\frac{(-1)^{r+j+1}}{\pi}
\arcsin\left(p_1(\ga h)q(\ga h)^{1/h}\right),
\end{equation}
which is equivalent to 
$\Omega_{n,j}(\ga,h)=0$.

Let us derive inequlities~\eqref{eq:rough_bounds_for_ga} 
from~\eqref{eq:upper_and_lower_bounds_for_al}.
An elementary analysis shows that
$0 \le \frac{\al}{2}-\tht(\al) \le \frac{1}{6}$
and $0\le p_1(\al) \le 1/2$
for every $\al$ in $[0,1/3]$.
Since $0\le \arcsin(t)\le \pi/6$ for every $t$ in $[0,1/2]$,
we get
\[
\ga_{n,j}
\le j - \frac{1}{2}+\frac{2}{3}
+ \frac{1}{6}+\frac{1}{6}
= j+\frac{1}{2},
\qquad
\ga_{n,j}
\ge
j - \frac{1}{2} - \frac{1}{6}
= j - \frac{2}{3}.
\]
Similar estimates could be obtained from
\eqref{eq:upper_and_lower_bounds_for_al}, though the additive constants would be slightly different.
\end{proof}

\begin{lem}
\label{lem:power_by_exp}
If $c_1,c_2>0$,
then the function
$x\mapsto x^{c_1}\enumber^{-c_2 x}$
is bounded on $[0,+\infty)$.
\end{lem}

\begin{proof}
Indeed, this function reaches its global maximum at the point $c_1/c_2$.
\end{proof}

In this section, we suppose that
$1\le j\le\log(n)$.
By~\eqref{eq:rough_bounds_for_ga},
we can assume that the variable $\ga$
satisfies
$\frac{1}{3}\le \ga\le \log\frac{1}{h}
+\frac{1}{2}$.

\begin{lem}
\label{lem:arcsin_p_q_power_near_zero}
As $h\to0$
and $1/3\le \ga\le \log(1/h)+1/2$,
the following asymptotic expansion holds:
\begin{equation}
\label{eq:arcsin_p_q_power_expansion_near_zero}
\begin{aligned}
    &\arcsin\bigl(p_1(\ga h)q(\ga h)^{1/h}\bigr)
    = \arcsin\left(\frac{1}{2}\enumber^{-\sqrt{3}\pi\ga}\right)
    + 
    \frac{2\pi^2\ga^2}{(4 \enumber^{2\sqrt{3}\pi\ga }-1)^{1/2}}\,h
    \\
    & \qquad
    +
    \frac{2(-\pi^2\ga^2 -\frac{2\sqrt{3}}{3}\pi^3\ga^3 + \pi^4\ga^4)}{(4 \enumber^{2\sqrt{3}\pi\ga }-1)^{1/2}}\,h^2
    +
    \frac{2\pi^4\ga^4}{(4\enumber^{2\sqrt{3}\pi\ga} -1)^{3/2}}h^2
    +
    O(\ga^6\,\enumber^{-\sqrt{3} \pi\ga}\,h^3).
\end{aligned}
\end{equation}
\end{lem}

\begin{proof}
Using the Taylor--Maclaurin expansions
\begin{align*}
\cos\left(\frac{\pi}{3}-\pi\al\right)
&=
\frac{1}{2}
+ \frac{\sqrt{3}\,\pi}{2}\,\al
-\frac{\pi^2}{4}\,\al^2
-\frac{\sqrt{3}\,\pi^3}{12}\,\al^3
+O(\al^4),
\\[0.5ex]
\sin\left(\frac{\pi}{3}-\pi\al\right)
&=
\frac{\sqrt{3}}{2}
-\frac{\pi}{2}\,\al
-\frac{\sqrt{3}\,\pi^2}{4}\,\al^2
+\frac{\pi^3}{12}\,\al^3
+O(\al^4),
\\[0.5ex]
\tan\left(\frac{\pi}{3}-\pi\al\right)
&=
\sqrt{3}
-4\pi\,\al
+4\sqrt{3}\,\pi^2\,\al^2
-\frac{40\pi^3}{3}\,\al^3
+O(\al^4),
\end{align*}
we easily get expansions of $p_1$ and $\log q$ around the origin:
\begin{equation}
\label{eq:p1_expansion}
p_1(\al)
=\frac{1}{2}
-\pi^2\,\al^2
+\frac{2\sqrt{3}\,\pi^3}{3}\,\al^3
+O(\al^4),
\end{equation}
\[
\log(q(\al))
=
- \sqrt{3}\,\pi\,\al
+ 2\pi^2\,\al^2
- \frac{4\sqrt{3}\,\pi^3}{3}\,\al^3
+ O(\al^4).
\]
In the last expansion,
we substitute $\ga h$ instead of $\al$.
Then, we divide both sides by $h$:
\[
\frac{1}{h}
\log(q(\ga h))
=
-\sqrt{3}\,\pi\,\ga
+ 2\pi^2\,\ga^2\,h
- \frac{4\sqrt{3}\,\pi^3}{3}\,\ga^3\,h^2
+ O(\ga^4\,h^3)
=
-\sqrt{3}\,\pi\,\ga
+ t.
\]
Here, for brevity, we have set
\[
t \eqdef
2\pi^2\,\ga^2\,h
- \frac{4\sqrt{3}\,\pi^3}{3}\,\ga^3\,h^2
+ O(\ga^4\,h^3).
\]
The assumption $\ga\le\log(1/h)+1/2$
implies that for every natural $k$,
the product $\ga^k h$ is bounded.
Therefore,
\begin{align*}
t^2
&=
4\pi^4\ga^4\,h^2
-\frac{16\sqrt{3}\,\pi^5}{3}\,\ga^5\,h^3
+\frac{16\pi^6}{3}\,\ga^6\,h^4
+O(\ga^6\,h^4)
+O(\ga^7\,h^5)
+O(\ga^8\,h^6)
\\
&=
4\pi^4\ga^4\,h^2
+O(\ga^5 h^3).
\end{align*}
Furthermore, using the same fact that 
$\ga^k h = O(1)$ for every natural $k$,
we easily get
\[
t = O(\ga^2 h), \qquad 
t^2 = O(\ga^4 h^2), \qquad 
O(t^3) = O(\ga^6\,h^3).
\]
Using these expansions and upper estimates of $t$, $t^2$, and $O(t^3)$,
we get an asymptotic expansion of $\exp(t)$:
\begin{align*}
\exp(t)
& = 1 + t + \frac{t^2}{2} + O(t^3)
\\
&
=1
+2\pi^2\ga^2\,h
-\frac{4\sqrt{3}\,\pi^3}{3}\ga^3\,h^2
+2\pi^4\ga^4\,h^2
+O(\ga^6 h^3).
\end{align*}
Now, we are ready to consider the power
$q(\ga h)^{1/h}$:
\begin{align}
q(\ga h)^{1/h}
&=
\exp\left(\frac{1}{h}\log(q(\ga h))\right)
=
\enumber^{-\sqrt{3}\,\pi\,\ga}
\exp(t)
\notag
\\
\label{q_power}
&=
\enumber^{-\sqrt{3}\,\pi\,\ga}
\left(
1
+2\pi^2\ga^2\,h
-\frac{4\sqrt{3}\,\pi^3}{3}\ga^3\,h^2
+2\pi^4\ga^4\,h^2
+O(\ga^6 h^3)
\right).
\end{align}
Using~\eqref{eq:p1_expansion}
and~\eqref{q_power},
after a simplification we get
\begin{equation}
\label{eq:pq_power}
p_1(\ga h)q(\ga h)^{1/h}  = \enumber^{-\sqrt{3}\,\pi\,\ga}
\left( \frac{1}{2} + \pi^2 \ga^2\,h + 
\left(- \pi^2\,\ga^2 -\frac{2\sqrt{3}}{3}\pi^3\ga^3 + \pi^4\ga^4\right) h^2
 + O(\ga^6 \,h^3)\right).
\end{equation}
We rewrite~\eqref{eq:pq_power}
in the form
$p_1(\ga h)q(\ga h)^{1/h}  
= \frac{1}{2}\enumber^{-\sqrt{3}\,\pi\,\ga} + u$,
where 
\begin{equation}
u \eqdef
\enumber^{-\sqrt{3}\,\pi\,\ga} \left( 
\pi^2 \ga^2\,h + 
\left(- \pi^2\,\ga^2 -\frac{2\sqrt{3}}{3}\pi^3\ga^3 + \pi^4\ga^4\right)
h^2 \right)
+ O(\ga^6 \, \enumber^{-\sqrt{3} \pi\ga}\,h^3).
\end{equation}
It is easy to see that
\[
u^2
=
\enumber^{-2\sqrt{3}\,\pi\,\ga} \pi^4\ga^4\,h^2 + O(\ga^6\,\enumber^{-2\sqrt{3} \pi\ga}\,h^3),
\]
\[
u = O(\ga^2\,\enumber^{-\sqrt{3} \pi\ga}\,h),
\qquad 
O(u^3) = O(\ga^4\,\enumber^{-3\sqrt{3} \pi\ga}\,h^3).
\]
Finally, we apply the Taylor expansion of
$\arcsin$:
\begin{equation}
\label{eq:arcsin_general_expansion}
\arcsin(c + u)
= \arcsin(c) + \frac{1}{(1 - c^2)^{1/2}} u + \frac{c}{2(1-c^2)^{3/2}}u^2 + O(u^3),
\end{equation}
with
$c=\frac{1}{2}\enumber^{-\sqrt{3}\pi\ga}$,
and after simplifying we get~\eqref{eq:arcsin_p_q_power_expansion_near_zero}.
\end{proof}

Figure~\ref{fig:approx_error_term}
shows the behavior of the error term in~\eqref{eq:arcsin_p_q_power_expansion_near_zero}.

\begin{figure}[htb]
\includegraphics{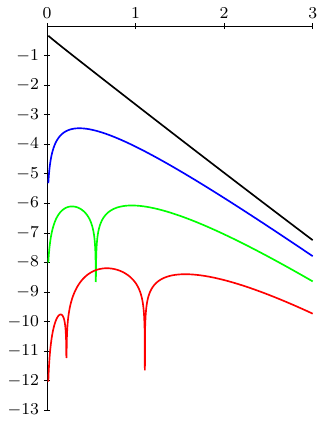}
\quad
\includegraphics{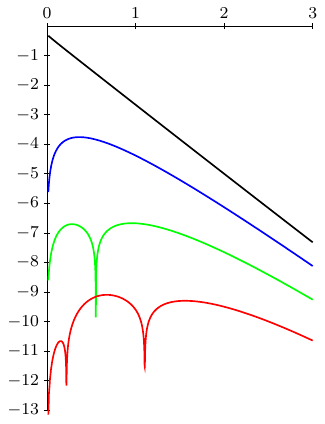}
\caption{Approximation of the function
$\ga\mapsto\arcsin\left(p_1(\ga h)p(\ga h)^{1/h}\right)$,
where $h=1/(n+2)$,
for $n=512$ (left) and $n=1024$ (right).
For the vertical axis,
we use the logarithmic scale ($\log_{10}$).
We show the exact value (black),
the error of the approximation with $1$ exact term
(blue),
$2$ exact terms (green),
and $3$ exact terms (red).
\label{fig:approx_error_term}
}
\end{figure}

\begin{lem}
\label{lem:transcendental_equation_contractive}
Let $r\in\{0,1,2\}$ and $j\in\bN$.
Define
$u_{r,j}\colon
\left[\frac{1}{3},+\infty\right)
\to\left[\frac{1}{3},+\infty\right)$
by
\begin{equation}
\label{eq:u_def}
u_{r,j}(y)
\eqdef
j-\frac{1}{2}+\frac{r}{3}
+\frac{(-1)^{r+j+1}}{\pi}
\arcsin\left(\frac{1}{2}\enumber^{-\sqrt{3}\,\pi\,y}\right).
\end{equation}
Then, $u_{r,j}$ is contractive,
and its Lipschitz coefficient is less than $1/7$.
\end{lem}

\begin{proof}
Since $j\ge 1$ and
\[
0
\le
\arcsin\left(\frac{1}{2}\enumber^{-\sqrt{3}\,\pi\,y}\right)
\le
\arcsin\frac{1}{2}=\frac{\pi}{6},
\]
it is easy to see that indeed
$u_{r,j}(y)\ge\frac{1}{3}$ for every $y$.
The derivative of $u_{r,j}$ is
\[
u_{r,j}'(y)
= \frac{(-1)^{r+j}\,\sqrt{3}\,\enumber^{-\sqrt{3}\,\pi y}}%
{2\sqrt{1-\frac{1}{4}\enumber^{-2\sqrt{3}\,\pi y}}}
= \frac{(-1)^{r+j}\,\sqrt{3}}%
{\sqrt{4\enumber^{2\sqrt{3}\,\pi y}-1}}.
\]
It can be absolutely bounded in the following way:
\[
|u_{r,j}'(y)|
\le
\left|u_{r,j}'\left(\frac{1}{3}\right)\right|
=
\frac{\sqrt{3}}{\sqrt{4\enumber^{2\pi\sqrt{3}/3}-1}}
<
0.142
< \frac{1}{7}.
\qedhere
\]
\end{proof}

We denote the fixed point of $u_{r,j}$
by $\ph_{r,j}$.

\begin{rem}
\label{rem:transcendental_equation}
The equation for $\ph_{r,j}$ can be rewritten in the following form, with trigonometric and exponential terms:
\begin{equation}
\label{eq:ph_equation_trig_exp}
\cos\left(\pi y-\frac{\pi r}{3}\right)
=\frac{(-1)^{r+1}}{2}\enumber^{-\sqrt{3}\,\pi\,y}.
\end{equation}
Obviously, \eqref{eq:ph_equation_trig_exp}
has an infinite set of positive solutions;
among them,
$\ph_{r,j}$ is the closest one to the point
$j-\frac{1}{2}+\frac{r}{3}$.
The following idea was suggested to us on a math forum, see
\cite{Stackexchange_answer}.
Making a complex exponential change of variables,
it is possible to transform~\eqref{eq:ph_equation_trig_exp}
to the form
\begin{equation}
\label{eq:Belkic}
x - b x^a - 1 = 0,
\end{equation}
where $a$ and $b$ are some complex constants.
Belki\'{c}~\cite[Section~8]{Belkic2019}
expressed the solutions of~\eqref{eq:Belkic}
in terms of a certain power series
related to the confluent Fox--Wright function.
Since the change of variables between~\eqref{eq:ph_equation_trig_exp}
and~\eqref{eq:Belkic} is not simple,
we are not developing this idea
in the present paper.
It would be interesting to
solve the equations of the form
$\exp(-ay)\cos(y)=b$
or $\exp(-ay)\sin(y)=b$
in terms of some power series,
without making complex changes of variables.
\end{rem}

Recall that the functions
$\om_{r,j,0}$, $\om_{r,j,1}$, and $\om_{r,j,2}$ are defined by~\eqref{eq:om0_def}, \eqref{eq:om1_def}, and \eqref{eq:om2_def}, respectively.

\begin{lem}[asymptotic expansion of $\Omega_{r,j}$]
\label{lem:Omega_expansion}
For every $r$ in $\{0,1,2\}$,
every $j$ in $\bN$,
every $h>0$ and $\ga$ satisfying
$\frac{1}{3}\le \ga\le j+\frac{1}{2}$,
the function $\Omega_{r,j}$ can be written in the form 
\[
\Omega_{r,j}(\ga,h)
=\om_{r,j,0}(\ga)
+\om_{r,j,1}(\ga)h
+\om_{r,j,2}(\ga)h^2
+R_{r,j}(\ga,h),
\]
where $R_{r,j}(\ga,h)=O(j^3 h^3)$.
\end{lem}

\begin{proof}
Applying~Lemma~\ref{lem:approximation_theta_near_origin} and taking into account that $\ga=O(j)$,
we get
\begin{equation}
\label{eq:tht_expansion_in_terms_of_ga}
\tht(\ga h)-\frac{\ga h}{2}
=
-\frac{2\sqrt{3}}{3}\pi \ga^2 h^2
+ O(\ga^3 h^3)
=
-\frac{2\sqrt{3}}{3}\pi \ga^2 h^2
+ O(j^3 h^3).
\end{equation}
Furthermore, due to Lemma~\ref{lem:power_by_exp},
the residue term in Lemma~\ref{lem:arcsin_p_q_power_near_zero} can be written as $O(h^3)$.
Therefore,
Lemma~\ref{lem:Omega_expansion} follows from~\eqref{eq:tht_expansion_in_terms_of_ga}
and~Lemma~\ref{lem:arcsin_p_q_power_near_zero}.
\end{proof}

In what follows, we use the first derivatives of $\om_{r,j,1}$ and $\om_{r,j,2}$.
They can be computed explicitly:
\begin{align}
\label{eq:om_0_D1}
\om_{r,j,0}'(\ga)
&=
1 - \frac{(-1)^{r+j}\sqrt{3}}{(4 \enumber^{2\sqrt{3}\pi\ga }-1)^{1/2}},
\\[1ex]
\label{eq:om_0_D2}
\om_{r,j,0}''(\ga)
&=
\frac{(-1)^{r+j}12\pi\enumber^{2\sqrt{3}\pi\ga}}{(4 \enumber^{2\sqrt{3}\pi\ga }-1)^{3/2}},
\\[1ex]
\label{eq:om_1_D1}
\om_{r,j,1}'(\ga)
&= \frac{(-1)^{r+j}4\pi}{(4 \enumber^{2\sqrt{3}\pi\ga }-1)^{1/2}}\left(\ga - \frac{2\sqrt{3}\pi\enumber^{2\sqrt{3}\pi\ga}}{4 \enumber^{2\sqrt{3}\pi\ga }-1} \ga^2 \right).
\end{align}

\begin{proof}[Proof of Theorem~\ref{thm:lambda_asympt_close_to_the_origin}.]
By Lemmas~\ref{lem:transcendental_equation_contractive}
and~\ref{lem:Omega_expansion},
we can apply the scheme from Proposition~\ref{prop:scheme_solve_equations_by_asymptotic_expansions3}.
Substituting $h=1/(n+2)$, we get
\begin{equation}
\label{eq:ga_asympt_expansion}
\ga_{n,j}
=\ph_{r,j}
+\frac{\psi_{r,j}}{n+2}
+\frac{\tau_{r,j}}{(n+2)^2}
+R_{3,n,j},
\end{equation}
where $\ph_{r,j},\psi_{r,j},\tau_{r,j}$
are defined before Theorem~\ref{thm:lambda_asympt_close_to_the_origin}
and
$R_{3,n,j}=O(j^3/(n+2)^3)$.
Let us explain this estimate of the residue term.
In our situation, the coefficients $c_1$, $c_2$, etc. in Proposition~\ref{prop:scheme_solve_equations_by_asymptotic_expansions3}
depend on the parameters $r$ and $j$
and can be bounded as
\[
c_{r,j,0,2}
=O(\enumber^{-\sqrt{3}\,\pi\,j}),
\quad
c_{r,j,0,3}
=O(\enumber^{-\sqrt{3}\,\pi\,j}),
\quad
c_{r,j,1,0}
=O(\enumber^{-\sqrt{3}\,\pi\,j}),
\quad
\text{etc.},
\]
\[
c_{r,j,3}=O(j^3),\qquad
c_{r,j,2}=O(j^3),\qquad
c_{r,j,1}=O(j^3),\qquad
k_{r,j}=O(j^3).
\]
By Lemma~\ref{lem:power_by_exp},
most of the products on the right-hand side of~\eqref{eq:eta3_bound}
admit uniform bounds,
and the most important term is
$c_{r,j,3}/\ell=O(j^3)$.

Dividing~\eqref{eq:ga_asympt_expansion} by $n+2$ we obtain
\[
\al_{n,j}
=
\frac{\ph_{r,j}}{n+2}
+\frac{\psi_{r,j}}{(n+2)^2}
+\frac{\tau_{r,j}}{(n+2)^3}
+O\left(\frac{j^3}{(n+2)^4}\right).
\]
Since $\la_{n,j}=G(\al_{n,j})$,
we get the result of Theorem~\ref{thm:lambda_asympt_close_to_the_origin}.
\end{proof}

\section{Numerical tests}
\label{sec:numerical_tests}

We have thoroughly tested all theorems from Section~\ref{sec:intro}
in SageMath~\cite{Sagemath}.
Our programs are freely available at GitHub, see~\cite{testtoeplitz1001}.

To verify Proposition~\ref{prop:charpol_coefficients_explicit} for moderate values of $n$,
we have used symbolic computations with univariate polynomials.
For other results, we used numerical computations with multiprecision arithmetic.

For all $n$ satisfying
$3\le n\le 2048$,
and all $j$ in $\{1,\ldots,m\}$,
we have calculated the eigenvalues $\la_{n,j}$ and the eigenvectors $v_{n,j}$,
using multiprecision computations with 512 binary digits
(approximately 154 decimal digits),
by applying the fixed point iteration method
(Theorem~\ref{thm:char_equation_as_fixed_point}) and formula~\eqref{eq:main_eigenvector},
and we have obtained small residues for the eigenpairs:
\begin{equation}
\label{eq:eigenpairs_numerical_test}
\max_{3\le n\le 2048}\
\max_{1\le j\le m}\
\frac{\|T_n(a) v_{n,j}-\la_{n,j} v_{n,j}\|_2}{\|v_{n,j}\|_2}
< 1.2 \cdot 10^{-148}.
\end{equation}

Next, we present some numerical tests for the asymptotic formulas.
Notice that general eigensolvers,
used with the machine precision arithmetic
or even with the multiprecision arithmetic,
typically give wrong results
for the matrices $T_n(a)$
of large order $n$;
see Figure~\ref{fig:symbol_and_wrong_eigenvalues_1024}.
Therefore, relying on~\eqref{eq:eigenpairs_numerical_test}, we treat the eigenvalues computed by Theorem~\ref{thm:char_equation_as_fixed_point} as the ``exact eigenvalues''
$\la_{n,j}$ in the tests below.

\begin{figure}[hbt]
\centering
\includegraphics{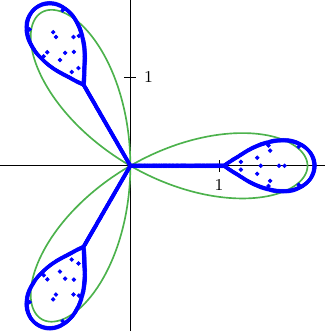}
\caption{The curve $\{a(t)\colon\ |t|=1\}$
(green)
and the wrong eigenvalues of $T_{1024}(a)$ 
(blue)
computed by a general eigensolver with multiprecision (512 bits).
\label{fig:symbol_and_wrong_eigenvalues_1024}}
\end{figure}

For each $n\ge 3$, we denote by $R_{1,n}$
the maximum absolute errors in
Theorem~\ref{thm:lambda_asympt_not_too_close_to_the_origin}
and~\ref{thm:lambda_asympt_close_to_the_origin}:
\[
R_{1,n}
\eqdef \max_{\lfloor \log n\rfloor \le j \le m}
|R_{1,n,j}|,
\qquad
R_{2,n}
\eqdef \max_{1\le j\le \lfloor \log n\rfloor}
|R_{2,n,j}|.
\]
We also consider the following ``maximum normalized errors'':
\[
\widetilde{R}_{1,n}
\eqdef \left(n+\frac{3}{2}\right)^3 R_{1,n},
\qquad
\widetilde{R}_{2,n}
\eqdef
\max_{1\le j\le \lfloor\log n\rfloor}
\frac{(n+2)^4\,|R_{2,n,j}|}{j^3}.
\]
Table~\ref{tab:errors} shows the values of $R_{1,n}$ and $R_{2,n}$ for some values of $n$.
In this table,
the values of $\widetilde{R}_{1,n}$
are bounded by $0.01$,
while the values of $\widetilde{R}_{2,n}$
are bounded by $228$.
In fact, these upper bounds hold not only for the values $n$ included in the table but for all $n$ with $3\le n\le 2048$.
Thereby, we have numerically verified
Theorems~\ref{thm:lambda_asympt_not_too_close_to_the_origin}
and~\ref{thm:lambda_asympt_close_to_the_origin}
for moderate values of $n$.
It is surprising for us
that the values of $\widetilde{R}_{1,n}$ are so small
and the values of $\widetilde{R}_{2,n}$ are much bigger, although bounded.
The main reason for the first phenomenon is that the derivatives of $\tht$ and
$\expterm_{n,j}$ are quite small
outside some neighborhood of $0$.
For the second phenomenon,
the main reason is that the derivatives of the functions $\tht$ and $\om_{r,j,k}$
(defined by~\eqref{eq:def_theta}
and~\eqref{eq:om0_def}--\eqref{eq:om2_def})
at the point $0$ are quite large:
$|\tht'''(0)|=|\om_{r,j,0}'''(0)|
\approx 79$.

\renewcommand{\arraystretch}{1.3}
\begin{table}[htb]
\[
\begin{array}{c|c|c|c|c|c}
 &
n=128 &
n=256 &
n=512 &
n=1024 &
n=2048
\\\hline
R_{1,n} &
2.86 \cdot 10^{-9} &
4.97 \cdot 10^{-10} &
6.71 \cdot 10^{-11} &
8.44 \cdot 10^{-12} &
1.06 \cdot 10^{-12}
\\
\widetilde{R}_{1,n} &
5.99 \cdot 10^{-3} &
8.33 \cdot 10^{-3} &
9.01 \cdot 10^{-3} &
9.06 \cdot 10^{-3} &
9.09 \cdot 10^{-3}
\\\hline
R_{2,n} &
2.95 \cdot 10^{-5} &
3.23 \cdot 10^{-6} &
4.45 \cdot 10^{-7} &
2.47 \cdot 10^{-8} &
2.92 \cdot 10^{-9}
\\
\widetilde{R}_{2,n} &
2.12 \cdot 10^2 &
1.14 \cdot 10^2 &
2.23 \cdot 10^2 &
1.27 \cdot 10^2 &
2.27 \cdot 10^2
\end{array}
\]
\caption{Maximum errors and ``normalized maximum errors''
in Theorems~\ref{thm:lambda_asympt_not_too_close_to_the_origin}
and~\ref{thm:lambda_asympt_close_to_the_origin}.
\label{tab:errors}
}
\end{table}

Figures~\ref{fig:errors_asympt_512}
and~\ref{fig:errors_asympt_1024_4096}
show the logarithmic errors of the asymptotic formulas in Theorems~\ref{thm:lambda_asympt_not_too_close_to_the_origin}
and \ref{thm:lambda_asympt_close_to_the_origin},
for some values of $n$.
We see that the asymptotic formula from Theorem~\ref{thm:lambda_asympt_not_too_close_to_the_origin} does not work well for the first values of $j$,
and the asymptotic formula from Theorem~\ref{thm:lambda_asympt_close_to_the_origin} is much better for the first values of $j$.

\begin{figure}[htb]
\centering
\includegraphics{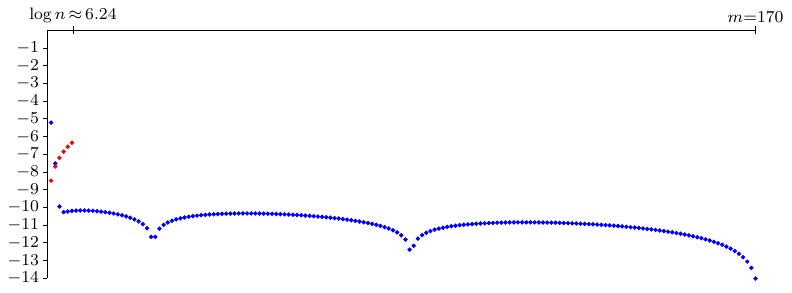}
\caption{$\log_{10}|R_{1,512,j}|$ for $1\le j\le m$ (blue points) and $\log_{10}|R_{2,512,j}|$ for $1\le j\le \log n$ (red points).
\label{fig:errors_asympt_512}}
\end{figure}

\begin{figure}[htb!]
\centering
\raisebox{-\height}{\includegraphics{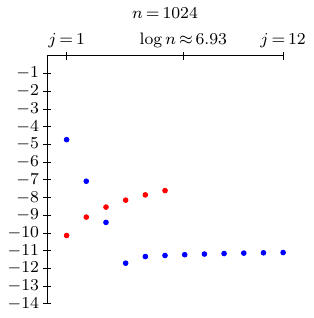}}%
\qquad\qquad
\raisebox{-\height}{\includegraphics{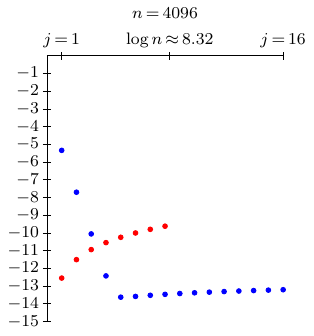}}
\caption{$\log_{10}|R_{1,n,j}|$
(blue points)
and $\log_{10}|R_{2,n,j}|$ (red points)
for $n=1024$ (left) and $n=4096$ (right),
for small values of $j$.
\label{fig:errors_asympt_1024_4096}}
\end{figure}

\section*{Data availability}

No data are associated
with this theoretical research.

\section*{Funding}

The first and third authors have been supported by SECIHTI Project (Mexico) ``Ciencia de Frontera''
FORDECYT-PRONACES/61517/2020.
The first author has also been supported by
the Ministry of Science and Higher Education of Russia,
Agreement 075-02-2026-1316.
Moreover, the third author has been  supported by IPN-SIP projects
(Instituto Polit\'{e}cnico Nacional, Mexico).

\smallskip
\section*{Information about the authors}

\medskip\noindent
Sergei M. Grudsky
(\myurl{https://orcid.org/0000-0002-3748-5449}),\\
email: grudsky@math.cinvestav.mx.
\\
Cinvestav,
Departamento de Matem\'{a}ticas
(Ciudad de M\'{e}xico, Postal Code 07360, Mexico);
Regional Mathematical Center,
Southern Federal University
(Bol’shaya Sadovaya Ulitsa 105/42,
Postal Code 344006,
Rostov-on-Don, Russia).

\bigskip\noindent
Rom\'{a}n Higuera-Garc\'{i}a
(\myurl{https://orcid.org/0009-0009-8644-2160}),
\\
email: rhiguerag@ciencias.unam.mx.
\\
Undergraduate student in Applied Mathematics at:
Universidad Nacional Aut\'{o}noma de M\'{e}xico (UNAM), Facultad de Ciencias
(Ciudad de M\'{e}xico, Postal Code 04510, Mexico).

\bigskip\noindent
Egor A. Maximenko
(\href{https://orcid.org/0000-0002-1497-4338}{https://orcid.org/0000-0002-1497-4338}),
\\
email: emaximenko@ipn.mx, egormaximenko@gmail.com.
\\
Instituto Polit\'{e}cnico Nacional,
Escuela Superior de F\'{i}sica y Matem\'{a}ticas
(Ciudad de M\'{e}xico,
Postal Code 07738, Mexico).

\bigskip\noindent
Fidel V\'{a}squez-Rojas
(\myurl{https://orcid.org/0009-0001-9221-0835}),
\\
email: fidel.vqzrojas@gmail.com.
\\
Graduate of Ph.D. program at:
Cinvestav,
Departamento de Matem\'{a}ticas
(Ciudad de M\'{e}xico, Postal Code 07360, Mexico).


\begin{thebibliography}{99}

\bibitem{Alexandersson2012}
P. Alexandersson:
Schur polynomials, banded Toeplitz matrices and Widom’s formula.
\emph{Electron. J. Combin.}
\textbf{19} (2012), no.~4, P22,
\doi{10.37236/2651}.

\bibitem{AptekarevKalyaginSaff2009}
A.~I.~Aptekarev,
V.~A.~Kalyagin,
E.~B.~Saff:
Higher-order three
term recurrences and asymptotics of multiple orthogonal polynomials.
\emph{Constr. Approx.}
\textbf{30} (2009), 175--223,
\doi{10.1007/s00365-008-9032-0}.

\bibitem{Assche2017}
W.~Van~Assche:
Mehler-Heine asymptotics for multiple orthogonal polynomials.
\emph{Proc. Amer. Math. Soc.}
\textbf{145} (2017), 303--314,
\doi{10.1090/proc/13214}.

\bibitem{BarreraBoettcherGrudskyMaximenko2018}
M.~Barrera,
A.~B\"{o}ttcher,
S.~M.~Grudsky,
E.~A.~Maximenko:
Eigenvalues of even very nice Toeplitz matrices can be unexpectedly erratic.
\emph{Oper. Theory: Adv. Appl.}
\textbf{268} (2018), 51--77,
\doi{10.1007/978-3-319-75996-8\_2}.

\bibitem{BatalshchikovGrudskyMalishevaMihalkovichRamirezStukopin2019}
A.~A.~Batalshchikov,
S.~M.~Grudsky,
I.~S.~Malisheva,
S.~S.~Mihalkovich,
E.~Ram\'{i}rez de Arellano,
V.~A.~Stukopin:
Asymptotics of eigenvalues of large symmetric Toeplitz matrices with smooth simple-loop symbols.
\emph{Linear Algebra Appl.}
\textbf{580} (2019), 292--335,
\doi{10.1016/j.laa.2019.06.017}.

\bibitem{BaxterSchmidt1961}
G.~Baxter,
P.~Schmidt:
Determinants of a certain class
of non-Hermitian Toeplitz matrices.
\emph{Math. Scand.}
\textbf{9} (1961), 122--128,
\doi{10.7146/math.scand.a-10630}.

\bibitem{Barry2022}
P. Barry:
\textit{Riordan Arrays: A Primer.}
Logic Press,
Kilcock, Co. Kildare, 2022.

\bibitem{Belkic2019}
D.~Belki\'{c}:
All the trinomial roots, their powers and logarithms from the Lambert series, Bell polynomials and
Fox--Wright function:
illustration for genome multiplicity in survival of irradiated cells.
\emph{J. Math. Chem.}
\textbf{57} (2019), 59--106,
\doi{10.1007/s10910-018-0985-3}.

\bibitem{BogoyaBoettcherGrudsky2012}
J.~M.~Bogoya,
A.~B\"{o}ttcher,
S.~M.~Grudsky:
Asymptotics of individual eigenvalues of a class of large Hessenberg Toeplitz matrices.
\emph{Oper. Theory Adv. Appl.}
\textbf{220} (2012), 77--95,
\doi{10.1007/978-3-0348-0346-5\_5}.

\bibitem{BogoyaBottcherGrudskyMaximenko2015}
J.~M.~Bogoya,
A.~B\"{o}ttcher,
S.~M.~Grudsky,
E.~A.~Maximenko:
Eigenvalues of Hermitian Toeplitz matrices with smooth simple-loop symbols.
\emph{J. Math. Anal. Appl.}
\textbf{422} (2015),
no.~2, 1308--1334,
\doi{10.1016/j.jmaa.2014.09.057}.

\bibitem{BogoyaGascaGrudsky2024}
M.~Bogoya,
J.~Gasca,
S.~Grudsky:
Eigenvalue asymptotic expansion for non-Hermitian tetradiagonal Toeplitz matrices with real spectrum.
\emph{J. Math. Anal. Appl.}
\textbf{531} (2024), 127816,
\doi{10.1016/j.jmaa.2023.127816}.

\bibitem{BogoyaGascaGrudsky2025}
M.~Bogoya,
J.~Gasca,
S.~M.~Grudsky:
Eigenvalues for a class of non-Hermitian tetradiagonal Toeplitz matrices.
\emph{J. Spectr. Theory}
\textbf{15} (2025), no.~1, 441--477,
\doi{10.4171/JST/538}.

\bibitem{BogoyaGrudskyMaximenko2017}
J.~M.~Bogoya,
S.~M.~Grudsky,
E.~A.~Maximenko:
Eigenvalues of Hermitian Toeplitz matrices generated by simple-loop symbols with relaxed smoothness.
\emph{Oper. Theory Adv. Appl.}
\textbf{259} (2017), 179--212,
\doi{10.1007/978-3-319-49182-0\_11}.

\bibitem{BGGK2021}
A.~B\"{o}ttcher,
J.~Gasca,
S.~M.~Grudsky,
A.~V.~Kozak:
Eigenvalue clusters of large tetradiagonal Toeplitz matrices.
\emph{Integr. Equ. Oper. Theory}
\textbf{93} (2021), 8,
\doi{10.1007/s00020-020-02619-z}.

\bibitem{BoettcherGrudsky2005}
A.~B\"{o}ttcher,
S.~Grudsky:
\emph{Spectral Properties of Banded Toeplitz Matrices}.
SIAM, Philadelphia, PA, 2005.

\bibitem{BoettcherGrudskyMaksimenko2010}
A.~Böttcher,
S.~M.~Grudsky,
E.~A.~Maksimenko:
Inside the eigenvalues of certain Hermitian Toeplitz band matrices.
\emph{J. Comput. Appl. Math.}
\textbf{233} (2010), no.~9, 2245--2264.
\doi{10.1016/j.cam.2009.10.010}.

\bibitem{BoettcherSilbermann1999}
A.~B\"{o}ttcher,
B.~Silbermann:
\emph{Introduction to Large Truncated Toeplitz Matrices}.
Springer, New York, NY, 1999.

\bibitem{Bruijn1958}
N.~G.~de~Bruijn:
\emph{Asymptotic Methods in Analysis}.
North-Holland Publishing Co.,
Amsterdam, 1958.

\bibitem{CoussementCoussementAssche2008}
E.~Coussement,
J.~Coussement,
W.~Van~Assche:
Asymptotic zero distribution for a class of multiple orthogonal polynomials.
\emph{Trans. Amer. Math. Soc.}
\textbf{360} (2008), 5571--5588,
\myurl{https://www.jstor.org/stable/20161479}.

\bibitem{DeiftItsKrasovsky2011}
P.~Deift,
A.~Its,
I.~Krasovsky:
Asymptotics of Toeplitz, Hankel, and Toeplitz $+$ Hankel determinants with Fisher-Hartwig singularities.
\emph{Ann. of Math.}
\textbf{174} (2011), 1243--1299,
\doi{10.4007/annals.2011.174.2.12}.

\bibitem{DuitsKuijlaars2008}
M.~Duits,
A.~B.~J.~Kuijlaars:
An equilibrium problem for the limiting eigenvalue distribution of banded Toeplitz matrices.
\emph{SIAM J. Matrix Anal. Appl.}
\textbf{30} (2008), 173--196,
\doi{10.1137/070687141}.

\bibitem{EiermannVarga1993}
M.~Eiermann,
R.~S.~Varga:
Zeros and local extreme points of
Faber polynomials associated with hypocloidal domains.
\emph{Electron. Trans. Numer. Anal.}
\textbf{1} (1993), 49--71,
\myurl{http://eudml.org/doc/118602}.

\bibitem{EkstromMeadon2023}
S.-E.~Ekstr\"{o}m,
D.~Meadon:
On the eigenvalues of Toeplitz matrices with two off-diagonals.
2023,
arXiv:2305.15107.

\bibitem{EkstromSerra2018}
S.-E.~Ekstr\"{o}m,
S.~Serra-Capizzano:
Eigenvalues and eigenvectors of banded Toeplitz matrices
and the related symbols.
\emph{Numer. Linear Algebra Appl.}
\textbf{25} (2018), no.~5, e2137,
\doi{10.1002/nla.2137}.

\bibitem{FonsecaVeerman2009}
C.~M. da Fonseca,
J.~J.~P.~Veerman:
On the spectra of certain directed paths.
\emph{Appl. Math. Lett.}
\textbf{22} (2009), 1351--1355,
\doi{10.1016/j.aml.2009.03.006}.

\bibitem{GaroniSerra2017}
C.~Garoni, S.~Serra-Capizzano:
\emph{Generalized Locally Toeplitz Sequences: Theory and Applications}, vol.~I.
Springer, Cham, Switzerland, 2017.

\bibitem{testtoeplitz1001}
S.~M.~Grudsky,
R.~Higuera-Garc\'{i}a,
E.~A.~Maximenko,
F.~Vasquez-Rojas:
Eigenvalues of the tetradiagonal Toeplitz matrices with diagonals
1, 0, 0, 1:
numerical tests.
2026.
\\
\myurl{https://github.com/EgorMaximenko/toeplitz\_matrices\_1001}.

\bibitem{HeSaff1994}
M.~X.~He, E.~B.~Saff:
The zeros of Faber polynomials for an
$m$-cusped hypocycloid,
\emph{J. Approx. Theory}
\textbf{78} (1994), 410--432,
\doi{doi.org/10.1006/jath.1994.1087}.

\bibitem{Hirschman1967}
I.~I.~Hirschman Jr.:
The spectra of certain Toeplitz matrices.
\emph{Illinois J. Math.}
\textbf{11}:1 (1967), 145--159,
\doi{10.1215/ijm/1256054792}.

\bibitem{Holmes2002}
G.~C.~Holmes:
The use of hyperbolic cosines in solving cubic polynomials.
\emph{Math. Gaz.}
\textbf{86}:507 (2022), 473--477,
\doi{10.2307/3621149}.

\bibitem{Stackexchange_answer}
IV\_
(https://math.stackexchange.com/users/292527/iv),
Are there some special functions to solve transcendental equations with exponential and trigonometric term
(in particular, $\exp(-x)=\cos x$)?,
\\
URL (version: 2025-08-18):
\myurl{https://math.stackexchange.com/q/5088332}.

\bibitem{MaximenkoMoctezuma2017}
E.~A.~Maximenko,
M.~A.~Moctezuma-Salazar:
Cofactors and eigenvectors of banded Toeplitz matrices:
Trench formulas via skew Schur polynomials.
\emph{Oper. Matrices}
\textbf{11} (2017), 
no.~4,
1149--1169,
\doi{10.7153/oam-2017-11-79}.

\bibitem{McMillen2009}
T.~McMillen:
On the eigenvalues of double band matrices.
\emph{Linear Algebra Appl.}
\textbf{431} (2009), 1890--1897,
\doi{10.1016/j.laa.2009.06.026}.

\bibitem{Rambour2026}
P.~Rambour:
Asymptotic behavior of the eigenvalues of Toeplitz matrices with even symbol.
\emph{Ann. Math. Blaise Pascal}
\textbf{33} (2026),
no.~1, 127--151,
\doi{10.5802/ambp.444}.

\bibitem{Sagemath}
The Sage Developers:
SageMath, the Sage Mathematics Software System,
Version 10.9 (2026),
\href{https://www.sagemath.org}{https://www.sagemath.org}.

\bibitem{SchmidtSpitzer1960}
P.~Schmidt,
F.~Spitzer:
The Toeplitz matrices of an arbitrary Laurent polynomial.
\emph{Math. Scand.}
\textbf{8} (1960), 15--38,
\doi{10.7146/math.scand.a-10588}.

\bibitem{Shapiro2022Riordan}
L.~Shapiro,
R.~Sprugnoli,
P.~Barry,
G.-S.~Cheon,
T.-X.~He,
D.~Merlini,
W.~Wang:
\emph{The Riordan Group and Applications.}
Springer, Cham, 2022.

\bibitem{Trench1985eig}
W.~F.~Trench:
On the eigenvalue problem for Toeplitz band matrices.
\emph{Linear Algebra Appl.}
\textbf{64} (1985), 199--214,
\doi{10.1016/0024-3795(85)90277-0}.

\bibitem{Ullman1967}
J.~L.~Ullman:
A problem of Schmidt and Spitzer.
\emph{Bull. Amer. Math. Soc.}
\textbf{73} (1987), 883--885,
\doi{10.1090/S0002-9904-1967-11826-3}.

\bibitem{Widom1958}
H.~Widom:
On the eigenvalues of certain Hermitian operators.
\emph{Trans. Am. Math. Soc.}
\textbf{88} (1958),
no.~2,
491--522,
\doi{10.2307/1993228}.

\bibitem{Widom1990}
H.~Widom:
Eigenvalue distribution of nonselfadjoint Toeplitz matrices and the asymptotics of Toeplitz determinants in the case of nonvanishing index.
\emph{Oper. Theory Adv. Appl.}
\textbf{48} (1990), 387--421.

\bibitem{Widom1994}
H.~Widom:
Eigenvalue distribution for nonselfadjoint Toeplitz matrices.
\emph{Oper. Theory Adv. Appl.}
\textbf{71} (1994), 1--8,
\doi{10.1007/978-3-0348-8543-0\_1}.

\bibitem{ZizlerZuidwijkTaylorArimoto2002}
P.~Zizler,
R.~A.~Zuidwijk,
K.~F.~Taylor,
S.~Arimoto:
A finer aspect of eigenvalue distribution of selfadjoint band Toeplitz matrices.
\emph{SIAM J. Matrix Anal. Appl.}
\textbf{24} (2002),
no.~1,
59--67,
\doi{10.1137/S089547989834915X}.

\end{thebibliography}
\end{document}